\documentclass[a4paper,10pt]{amsart}
\usepackage[a4paper,tmargin=40mm,bmargin=35mm,hmargin=30mm]{geometry}
\usepackage[utf8]{inputenc}
\usepackage[T1]{fontenc}

\usepackage{amsmath}
\usepackage{amssymb}
\usepackage{amsthm}
\usepackage{mathtools}

\usepackage{eucal} % for nicer calligraphic letters
\usepackage{mathdots}
\usepackage{hyperref}
\usepackage[capitalise]{cleveref}
\usepackage{enumitem}

\usepackage{tikz-cd} % much nicer arrows than xymatrix
\usepackage{adjustbox}

\newtheorem{thm}{Theorem}[section]
\newtheorem*{uthm}{Theorem}
\newtheorem{prop}[thm]{Proposition}
\newtheorem{cor}[thm]{Corollary}
\newtheorem{lemma}[thm]{Lemma}
\newtheorem{setting}[thm]{Setting}
\theoremstyle{definition}
\newtheorem{dfn}[thm]{Definition}
\newtheorem{rmk}[thm]{Remark}
\newtheorem{ex}[thm]{Example}

\DeclareMathOperator{\Hom}{\mathsf{Hom}}
\DeclareMathOperator{\End}{\mathsf{End}}
\DeclareMathOperator{\Ext}{\mathsf{Ext}}
\DeclareMathOperator{\HOM}{\mathcal{H}om}
\DeclareMathOperator{\RHom}{\mathbb{R}\mathsf{Hom}}
\DeclareMathOperator{\Cat}{\mathsf{Cat}}
\DeclareMathOperator{\CAT}{\mathsf{CAT}}
\newcommand{\Acal}{\mathcal{A}}

\newcommand{\Ccal}{\mathcal{C}}

\newcommand{\Ecal}{\mathcal{E}}
\newcommand{\Fcal}{\mathcal{F}}
\newcommand{\Gcal}{\mathcal{G}}
\newcommand{\Hcal}{\mathcal{H}}

\newcommand{\Kcal}{\mathcal{K}}
\newcommand{\Lcal}{\mathcal{L}}

\newcommand{\Scal}{\mathcal{S}}
\newcommand{\Tcal}{\mathcal{T}}
\newcommand{\Ucal}{\mathcal{U}}
\newcommand{\Vcal}{\mathcal{V}}

\newcommand{\Xcal}{\mathcal{X}}
\newcommand{\Ycal}{\mathcal{Y}}

\newcommand{\Rbb}{\mathbb{R}}

\newcommand{\Tbb}{\mathbb{T}}
\newcommand{\Zbb}{\mathbb{Z}}
\newcommand{\C}{\mathsf{C}}
\newcommand{\D}{\mathsf{D}}
\newcommand{\co}{\mathsf{co}}

\newcommand{\K}{\mathsf{K}}
\newcommand{\Kac}{\mathsf{K}_\mathsf{ac}}

\newcommand{\Kcoac}{\mathsf{K}_\mathsf{coac}}
\newcommand{\inj}[1]{\mathsf{Inj}(#1)}
\newcommand{\fpinj}[1]{\mathsf{fpInj-}#1}
\newcommand{\fp}[1]{\mathsf{fp}(#1)} 
\newcommand{\Mod}[1]{\mathsf{Mod}\mbox{-}#1}
\newcommand{\Qcoh}[1]{\mathsf{Qcoh}\mbox{-}#1}
\newcommand{\coh}[1]{\mathsf{coh}\mbox{-}#1}
\newcommand{\Per}[1]{\mathsf{Per}\mbox{-}#1}
\renewcommand{\mod}[1]{\mathsf{mod}\mbox{-}#1}
\newcommand{\real}{\mathsf{real}}

\newcommand{\Spec}[1]{\mathsf{Spec}(#1)}
\newcommand{\Supp}{\mathsf{Supp}}
\newcommand{\Lex}{\mathsf{Lex}}
\newcommand{\Ab}{\mathsf{Ab}}
\newcommand{\Ass}[1]{\mathsf{Ass}(#1)}
\newcommand*{\Perp}[1]{{}^{\perp_{#1}}}
\newcommand{\hocolim}{\mathsf{hocolim}}
\newcommand{\FPinf}{\mathsf{FP}_\infty}

\newcommand{\Prod}{\mathsf{Prod}}
\newcommand{\Add}{\mathsf{Add}}
\newcommand{\ann}[1]{\mathsf{ann}(#1)}
\newcommand{\pp}{\mathfrak{p}}
\newcommand{\qq}{\mathfrak{q}}
\newcommand{\mm}{\mathfrak{m}}
\newcommand{\sing}{\mathsf{sg}}

\newcommand{\DS}{\mathsf{S}}
\newcommand{\smallrecoll}[3]{
	\begin{tikzcd}[column sep=1pc, row sep=1pc,ampersand replacement=\&]
		\ensuremath{#1} \arrow{r} \& \ensuremath{#2} \arrow[shift left]{l} \arrow[shift right]{l} \arrow{r} \&
			\ensuremath{#3}  \arrow[shift left]{l} \arrow[shift right]{l} \end{tikzcd}
}
\newcommand{\toeq}{\xrightarrow{\sim}}
\newcommand{\Abfs}[1]{\mathsf{Ab4}^*\mathsf{-}#1}

\newcommand{\newterm}[1]{\textbf{#1}}

\title[Singular equivalences to locally coherent hearts]{Singular equivalences to locally coherent hearts of commutative noetherian rings}
\author{Michal Hrbek}
\address[M. Hrbek]{Institute of Mathematics of the Czech Academy of Sciences, \v{Z}itn\'{a} 25, 115 67 Prague, Czech Republic}
\email{hrbek@math.cas.cz}
\author{Sergio Pavon}
\address[S. Pavon]{Dipartimento di Matematica ``Tullio Levi-Civita'', Università degli Studi di Padova, Via Trieste, 63, 35121 Padova, Italy}
\email{sergio.pavon@math.unipd.it}

\thanks{M. Hrbek was supported by the GAČR project 20-13778S and RVO: 67985840.
S. Pavon was partially supported by Project BIRD163492/16 and Research programme
DOR1828909; he wishes to thank the Institute of Mathematics of the Czech Academy of
Sciences for its hospitality while working on this project.}
\begin{document}
\begin{abstract}
	We show that Krause's recollement exists for any locally coherent
	Grothendieck category whose derived category is compactly generated.
	As a source of such categories, we consider the hearts of intermediate and
	restrictable $t$-structures in the derived category of a commutative
	noetherian ring. We show that the induced tilting object over such a heart
	gives rise to an equivalence between the two Krause's recollements, and in
	particular, to a singular equivalence.
\end{abstract}
\maketitle
\tableofcontents
\section*{Introduction}

When working in the derived category $\D(\Gcal)$ of an abelian category $\Gcal$,
an important tool used since the beginning of the theory was its relation with the homotopy
category $\K(\inj\Gcal)$ of complexes with injective terms, via the construction
of injective resolutions. It was an early observation, though, that not all
complexes of injectives can be used as resolutions: the correct ones to consider
are the dg-injective ones, as explained by Spaltenstein \cite{Spa}. 

Nevertheless, a decade and a half later
Krause \cite{K05} showed that also the whole homotopy category of injectives $\K(\inj\Gcal)$
deserves an attention of its own. He
proved that if $\Gcal$ is a locally noetherian Grothendieck category such that
its derived category is compactly generated (e.g. the category of quasicoherent
sheaves over a noetherian scheme, or the category of modules over a noetherian
ring) then $\K(\inj \Gcal)$ is also compactly generated,
and the Verdier localization functor $Q: \K(\inj \Gcal) \rightarrow
\D(\Gcal)$, together with its right adjoint $\D(\Gcal)\to
\K(\inj\Gcal)$ which computes dg-injective resolutions, can be completed to a
\emph{recollement} of triangulated categories
\[\smallrecoll{\Kac(\inj \Gcal)}{\K(\inj \Gcal)}{\D(\Gcal)}.\]
Here, $\Kac(\inj \Gcal)$ is the full subcategory of $\K(\inj \Gcal)$ consisting
of the acyclic complexes, called
the \emph{stable derived category} of $\Gcal$ in \cite{K05}. An important point is
that the recollement renders $\Kac(\inj \Gcal)$ compactly generated as well; and
in fact, its compact objects form a category equivalent (up to retracts) to the
\emph{singularity category} $\D^\sing(\Gcal) = \D^b(\fp \Gcal)/\D^c(\Gcal)$, an
important concept introduced by Buchweitz \cite{buch-87} and Orlov
\cite{orlo-06} in order to measure the failure of a
scheme to be regular.

Another decade later, Šťovíček \cite{St} showed that we obtain a similar
picture if we replace the locally noetherian condition by the much more general
one of being locally coherent. This wider class of Grothendieck categories is of
considerable interest, as it encompasses, for example, many categories arising from
triangulated purity theory \cite{beli-00}, localizations of module categories
\cite{herzog, Kr97}, as well as some categories coming from tilting theory on which
we focus below.
This generalization required employing a rather
different approach including model category techniques,
viewing $\K(\inj \Gcal)$ as the \emph{coderived category} of $\Gcal$ in
the sense of Becker \cite{Bec14}. To obtain Krause's recollement in the
locally coherent setting however, an additional hypothesis was used in \cite{St}
--- $\Gcal$ is assumed to have a set of finitely presented generators which are
of finite projective dimension. The first aim of the present paper is
to show, in Section 2, that this assumption can be significantly weakened,
obtaining a full generalisation of Krause's result:

\begin{uthm}[\ref{theorem:recollement}]
	Let $\Gcal$ be a locally coherent Grothendieck category whose derived
	category is compactly generated. Then Krause's recollement exists for
	$\Gcal$.
\end{uthm}

%is unnecessary and the Krause's recollement exists
%for any locally coherent Grothendieck category with compactly generated derived
%category (\cref{theorem:recollement}).
In particular, we define the singularity category of these Grothendieck
categories. We proceed to show (\cref{prop-sing}) that as soon as $\Gcal$ satisfies a
mild additional condition, the vanishing of the singularity category has the
expected homological interpretation, in terms of finitely presented objects
having finite projective dimension. In the case of the category of quasicoherent sheaves over a separated coherent scheme $X$, our notion of singularity category coincides with the Orlov's quotient $\D^b(\coh X)/\Per X$ (\cref{cor:scheme}). In \cref{subsec:small}, we use Roos' correspondence between skeletally small and locally coherent abelian categories to define the singularity category of a large class of small abelian categories, and compare this to Orlov's construction of a singularity category of a general triangulated category.

As anticipated above, in Sections~3 and~4 we turn our attention to another
rich source of locally coherent Grothendieck categories to which the theorem
applies: the unbounded derived category $\D(\Mod R)$ of a commutative noetherian
ring $R$. Here, our categories of interest appear as the hearts of restrictable
$t$-structures, i.e. those inducing $t$-structures on the bounded derived
category $\D^b(\mod R)$ of finitely presented modules, by restriction.

These hearts are indeed locally coherent \cite[Corollary~4.2]{MZ}, while
they are rarely locally noetherian \cite[Proposition~5.6]{Laking}, and they
might fail to satisfy the additional hypothesis used in \cite{St} (as shown in \cref{example:hyp7.1}).
On the other hand, their derived categories have recently been shown to be
always equivalent to $\D(\Mod R)$ \cite[Theorem~6.16]{PV2} --- in particular,
they are compactly generated, and \cref{theorem:recollement} can be applied to obtain
Krause's recollement (\cref{cor:restr-recollement}). The existence of these
derived equivalences, as well as the particularly nice form they can be
expressed in (\cref{lemma:rhom-equivalence}), is our main reason for
restricting ourselves to the
commutative world, since an analogue for non-commutative rings is not available
so far.

We remark that there is no shortage of restrictable t-structures
(see \cref{ubiqutous}); for example, whenever $R$ admits a dualizing
complex then the \emph{Cohen-Macaulay t-structure} in $\D^b(\mod R)$ in the
sense of \cite{AJS10} extends to such a t-structure in $\D(\Mod R)$. More
generally, in this situation restrictable $t$-structures can be constructed by
\emph{sp-filtrations} satisfying the \emph{weak Cousin condition}
\cite[Theorem~6.9]{AJS10}.

%a recollement can be constructed using the theorem above,
%because the derived categories of these hearts are known to be equivalent to
%$\D(\Mod R)$ --- in particular, compactly generated.
%
%any such heart $\Hcal$ is induced by a cotilting object of
%$\D(\Mod R)$, and its derived category is known to be triangle equivalent to
%$\D(\Mod R)$ --- in particular, it is compactly generated.
%This allows to apply the theorem and construct Krause's recollement.
%
Now,
%We then ask the natural question of
what is the comparison between the recollements for
$\Mod R$ and the heart $\Hcal$ of such a $t$-structure?
In Section~4 we answer this question with the following
\begin{uthm}[\ref{theorem:eor}]
	Let $R$ be a commutative noetherian ring and $\Hcal$ the heart of an
	intermediate restrictable $t$-structure in $\D(\Mod R)$. Then there is an
	equivalence of recollements
	\[\begin{tikzcd}[column sep=1pc, row sep=1pc]
		\Kac(\inj\Hcal) \arrow{d}{\cong} \arrow{r} &
			\K(\inj\Hcal) \arrow{d}{\cong} \arrow[shift left]{l} \arrow[shift right]{l} \arrow{r} &
			\D(\Hcal) \arrow{d}{\cong} \arrow[shift left]{l} \arrow[shift right]{l} \\
		\Kac(\inj{R}) \arrow{r} &
			\K(\inj{R}) \arrow[shift left]{l} \arrow[shift right]{l} \arrow{r} &
			\D(\Mod{R}). \arrow[shift left]{l} \arrow[shift right]{l}
	\end{tikzcd}\]
\end{uthm}
In particular, we obtain an equivalence between the singularity categories and
stable derived categories of $\Mod R$ and $\Hcal$ (\cref{cor:sing-eq}, \cref{theorem:eor}). 

We conclude this introduction by mentioning that in the literature (see e.g.
\cite[Lemma~4.1]{IW}) there have been considered singular equivalences induced
by derived equivalences between the bounded derived categories of coherent
objects over schemes or rings. In our situation however, the derived
equivalences come from the ``large'' tilting theory, as developed in
\cite{NSZ,PV,Vi}. This forces us to use different techniques to obtain
the singular equivalence, including working with an enhancement of the unbounded
derived categories in the form of stable derivators.

\subsection*{Acknowledgement} We are grateful to Leonid Positselski and Martin Kalck for very
useful discussions concerning the paper, and to Steffen Koenig for some
suggestions regarding exposition.

\section{Preliminaries}

In this section we recall from the theory of triangulated categories the various
concepts we will need later: compact objects, $t$-structures, recollements, derived and
coderived categories, derivators and realization functors.

\subsection{Compact objects in triangulated categories}
A major role in our discussion will be played by compact objects.

\begin{dfn}
	Let $\Tcal$ be a triangulated category. An object $C\in\Tcal$ is said to
	be \newterm{compact} if, for every family $(X_i \mid i\in I)$ of
	objects whose coproduct exists in $\Tcal$, the canonical morphism
	\[\coprod_{i\in I}\Hom_\Tcal(C,X_i) \to \Hom_\Tcal(C, \coprod_{i\in I} X_i)\]
	is an isomorphism. The full subcategory of compact objects of $\Tcal$
	(which is a thick subcategory) will be denoted by $\Tcal^c$. If $\Tcal$
	has all coproducts, it is said to be \newterm{compactly generated} if it
	coincides with its smallest triangulated subcategory closed under
	coproducts and containing $\Tcal^c$.
\end{dfn}

We will often employ a \emph{dévissage} argument, which is a standard tool. For
the convenience of the reader, we spell out once the application we will use
the most.

\begin{lemma}[Double \emph{dévissage}]\label{lemma:double-devissage}
	Let $\Tcal,\Scal$ be compactly generated triangulated categories, and
	$F\colon \Tcal\to \Scal$ a triangle functor. Assume that $F$ preserves
	coproducts, and that it restricts to an equivalence $\Tcal^c\to
	\Scal^c$. Then $F$ is an equivalence.
\end{lemma}

\begin{proof}
	We first prove that $F$ is fully faithful.
	For every $X,Y$ in $\Tcal$, consider the natural map 
	\[ \eta_{X,Y}\colon \Hom_\Tcal(X,Y)\to \Hom_\Scal(FX,FY)\]
	induced by $F$. Let $\Ycal\subseteq\Tcal$ be the full subcategory
	of the objects $Y$ such that $\eta_{C,Y}$ is an
	isomorphism for every $C\in\Tcal^c$. It is easily seen to be
	triangulated; moreover, since $F$ preserves coproducts and the objects
	$C$ and $FC$ are compact in $\Tcal$ and $\Scal$ respectively, $\Ycal$ is
	also closed under coproducts. By hypothesis, $\Tcal^c\subseteq\Ycal$,
	and therefore $\Ycal=\Tcal$. Now, let $\Xcal\subseteq\Tcal$ be the full
	subcategory of the objects $X$ such that $\eta_{X,Y}$ is an isomorphism
	for every $Y\in\Tcal$. Again, it is triangulated; this time it is also
	automatically closed under coproducts, and by the previous discussion
	$\Tcal^c\subseteq\Xcal$. We conclude that $\Xcal=\Tcal$, i.e. that $F$
	is fully faithful. Now, consider the essential image of $F$ in $\Scal$.
	Since $F$ is a full triangle functor, its image is a triangulated
	subcategory (fullness is needed for closure under extensions). Moreover,
	it is also closed under coproducts, because $F$ preserves them, and
	contains $\Scal^c$, by hypothesis. We deduce that $F$ is also
	essentially surjective, i.e. an equivalence.
\end{proof}
\subsection{$t$-structures}\label{prelim:t-str}
\cite{BBD} Let $\Tcal$ be a triangulated category. A pair $\Tbb = (\Ucal,\Vcal)$ of full subcategories of $\Tcal$ is a \newterm{$t$-structure} provided that the following axioms hold:
\begin{enumerate}
	\item[(t-1)] $\Hom_{\Tcal}(\Ucal,\Vcal) = 0$,
	\item[(t-2)] $\Ucal$ is closed under the suspension functor, and
	\item[(t-3)] for any $X \in \Tcal$ there is a triangle $U \rightarrow X \rightarrow V \rightarrow U[1]$ with $U \in \Ucal$ and $V \in \Vcal$.  
\end{enumerate}

We call the subcategory $\Ucal$ (resp. $\Vcal$) the \newterm{aisle} (resp. the
\newterm{coaisle}) of the $t$-structure $\Tbb$. It follows from the axioms that
$\Ucal = \Perp{0}\Vcal = \{X \in \Tcal \mid \Hom_\Tcal(X,\Vcal) = 0\}$ and
$\Vcal = \Ucal\Perp{0}$, and so any $t$-structure is uniquely determined by its
aisle or by its coaisle. The triangle from the axiom (t-3) is unique and
functorial. In fact, the triangle is isomorphic to a triangle of the form
$\tau_\Ucal(X) \rightarrow X \rightarrow \tau_\Vcal(X) \rightarrow
\tau_\Ucal(X)$, where $\tau_\Ucal: \Tcal \rightarrow \Ucal$ (resp. $\tau_\Ucal:
\Tcal \rightarrow \Ucal$) is the right (resp. left) adjoint to the inclusion of
the aisle (resp. coaisle) into $\Tcal$. The \newterm{heart} of the $t$-structure
$\Tbb$ is defined as $\Hcal = \Ucal \cap \Vcal[1]$ and it is an abelian category
with the exact structure induced by the triangles of $\Tcal$ whose terms belong
to $\Hcal$.

Assuming that $\Tcal$ has a suitable enhancement, see \cite[\S3, Theorem
3.11]{PV} or \cite[\S4]{Vi}, there exists a \newterm{(bounded) realization functor}
associated to the $t$-structure $\Tbb$, i.e. a 
triangle functor $\real^b_{\Tbb}:
\D^b(\Hcal) \rightarrow \Tcal$ which extends the inclusion $\Hcal \subseteq
\Tcal$. 
Realization functors are not uniquely determined in general, but as shown in
\cite[Proposition~3.17]{PV}, bounded derived equivalences of abelian categories
are always of the form $\real^b_\Tbb$ for a suitable $t$-structure $\Tbb$, up to
an equivalence of abelian categories.

A $t$-structure $\Tbb = (\Ucal,\Vcal)$ is called \newterm{stable} if its aisle
$\Ucal$ (equivalently, its coaisle $\Vcal$) is a triangulated subcategory of
$\Tcal$. The aisles of stable $t$-structures are precisely the coreflective
thick subcategories of $\Tcal$ (\cite[Proposition~4.9.1]{K10}). Such
subcategories are automatically \newterm{localising}, i.e. triangulated and
closed under existing coproducts.

\subsection{Recollements and their equivalences}\label{sec:recollements}

\cite{BBD} Let $\Ucal,\Vcal,\Tcal$ be triangulated categories. A
\newterm{recollement} (of $\Tcal$) is a diagram of triangle functors
\begin{equation}\label{eq-recollement}
\begin{tikzcd}[column sep=huge]
	\Vcal \arrow{r}{i_*} &
	\Tcal \arrow[shift left,  bend left] {l}       {i^!}
			   \arrow[shift right, bend right]{l}[above]{i^*}
			   \arrow{r}{j^*} &
	\Ucal \arrow[shift left,  bend left] {l}       {j_*}
			  \arrow[shift right, bend right]{l}[above]{j_!}
\end{tikzcd}
\end{equation}
such that:
\begin{enumerate}
	\item[(i)] $(i^*,i_*,i^!)$ and $(j_!,j^*,j_*)$ are adjoint triples,
	\item[(ii)] $i_*, j_!, j_*$ are fully faithful,
	\item[(iii)] $\mathrm{Im}(i_*) = \mathrm{Ker}(j^*)$.
\end{enumerate}

We say that two recollements $\smallrecoll{\Vcal}{\Tcal}{\Ucal}$ and
$\smallrecoll{\Vcal'}{\Tcal'}{\Ucal'}$ are \newterm{equivalent} if there are
triangle equivalences $F: \Tcal \rightarrow \Tcal'$ and $G: \Ucal \rightarrow
\Ucal'$ such that the diagram
\[ \begin{tikzcd}
	\Tcal  \arrow{d}{\cong}[swap]{F}\arrow{r}{j^*} & \Ucal \arrow{d}{\cong}[swap]{G}\\
	\Tcal' \arrow{r}{j^{*'}} & \Ucal'
\end{tikzcd}
\]
is commutative (up to a natural equivalence). It follows from \cite{PS88} that
this situation is enough to induce a triangle equivalence $H: \Vcal \rightarrow
\Vcal'$ and the commutativity of all of the six possible squares corresponding
to the six different functors from the definition of recollement
\cref{eq-recollement}.

It is well-known that any recollement as in \cref{eq-recollement} induces a
\newterm{(stable) TTF triple} $(j_!\Ucal,i_*\Vcal,j_*\Ucal)$, that is, a pair of
two adjacent (automatically stable) $t$-structures $(j_!\Ucal,i_*\Vcal)$ and
$(i_*\Vcal,j_*\Ucal)$. In fact, this assignment yields a bijective
correspondence between equivalence classes of recollements of $\Tcal$ and TTF
triples in $\Tcal$.

\subsection{Categories of complexes and the coderived category.}

Let $\Gcal$ be an abelian category. We will deal with many categories whose
objects are complexes with terms in $\Gcal$, so we proceed to fix the notation,
in order to recall some less known definitions and to point out the relations among them.

As usual, $\C(\Gcal)$ denotes the category of complexes and cochain maps.
Inside $\C(\Gcal)$, one can consider the acyclic complexes, and among them the
contractible ones. By forming the quotient over the contractible complexes, one
obtains the homotopy category $\K(\Gcal)$ of $\Gcal$. Inside $\K(\Gcal)$ there
are again the acyclic complexes, whose subcategory is denoted by $\Kac(\Gcal)$.
The derived category $\D(\Gcal)$ of $\Gcal$ is defined as the Verdier
localisation $\K(\Gcal)/\Kac(\Gcal)$, and in all the occurrences in this paper
this construction will result in an honest (triangulated) category. The
localisation functor will be denoted by $Q\colon \K(\Gcal)\to \D(\Gcal)$.
Notice that when $\Gcal$ has exact coproducts, $Q$ commutes with
coproducts. We denote by $\D^b(\Gcal)$ the bounded derived category of $\Gcal$
--- the full triangulated subcategory of $\D(\Gcal)$ consisting of objects whose
cohomology vanishes in all but finitely many degrees.

Now let $\Gcal$ be a Grothendieck abelian category. Inside $\K(\Gcal)$ there is
the subcategory $\K(\inj\Gcal)$ of complexes with injective terms. Its left
$\Hom$-orthogonal is the subcategory $\Kcoac(\Gcal)$ of \newterm{coacyclic
objects}. These are equivalently defined in $\C(\Gcal)$ as those complexes $X$
such that $\Ext^1_{\C(\Gcal)}(X,Y)=0$ for every complex $Y$ with injective terms
(see \cite[Definition~6.7]{St}, and \cite{Bec14,Pos11} for the original
definitions). The pair of subcategories $(\Kcoac(\Gcal),\K(\inj\Gcal))$ is a
stable $t$-structure in $\K(\Gcal)$; the corresponding right truncation will be
denoted by $I_\lambda\colon \K(\Gcal)\to \K(\inj\Gcal)$ (see \cite[Corollary 7
and Example 5]{K11}). The \newterm{coderived category} (in Becker's sense) of
$\Gcal$ is defined as the Verdier localisation $\D^\co(\Gcal) :=
\K(\Gcal)/\Kcoac(\Gcal)$, and it is equivalent to $\K(\inj\Gcal)$ via the
functor induced by $I_\lambda$. Coacyclic complexes are in particular acyclic
(otherwise they would have a non-zero morphism to the injective envelopes of
their non-zero cohomologies), so there is a localisation $\D^\co(\Gcal)\to
\D(\Gcal)$, which corresponds to the restriction of $Q$ after identifying $\D^\co(\Gcal)\cong
\K(\inj\Gcal)$.

%Since $\Gcal$ is Grothendieck, the restriction $Q\colon \K(\inj\Gcal)\to
%\D(\Gcal)$ has a right adjoint $Q_r\colon \D(\Gcal)\to \K(\inj\Gcal)$, whose
%essential image is the subcategory of \newterm{dg-injective} complexes
%\cite{Ser}. This subcategory is the coaisle of a stable $t$-structure in
%$\K(\inj\Gcal)$ whose aisle is the subcategory $\Kac(\inj\Gcal)$ of acyclic
%complexes with injective terms. In the next section we will see that if $\Gcal$
%is in addition locally coherent and $\D(\Gcal)$ is compactly generated, this
%$t$-structure fits in a recollement.

\begin{rmk}
	There is a different definition of a coderived category in the literature,
	which is due to Positselski \cite{Pos11}. The two definitions are known to
	coincide in many situations, for example if the underlying Grothendieck
	category is locally noetherian \cite[\S3.7]{Pos11}, but it seems to be an
	open question even for module categories whether they coincide in general
	(see e.g. \cite[Example 2.5(3)]{Pos20}). However, as we will see in
	Corollary~\ref{cor:poscoderived}, for the locally Grothendieck categories we
	are most interested in, that is the hearts of intermediate restrictable
	$t$-structure over commutative noetherian rings, the two definitions of a
	coderived category are indeed equivalent, and so there is no need to
	distinguish them.
\end{rmk}

\subsection{Derivators}\label{sec:derivators}

For some of our arguments to work correctly, we will need to consider
$\D(\Gcal)$ enhanced with the structure of a stable derivator. For basics about
the standard derivator of a Grothendieck category which covers most of what is
needed in our application see e.g. \cite{St2} or \cite[Appendix]{HN} and the
references therein. Here we recall only some particular aspects and terminology
of the theory.

Let $\CAT$ denote the large 2-category of all categories, $\Cat$ denote the
2-category of all small categories and $\Gcal$ be a Grothendieck category. For
any $I \in \Cat$ we consider the Grothendieck category $\Gcal^I$ of all
$I$-shaped diagrams in $\Gcal$, that is, of all functors $I \rightarrow \Gcal$.
Since $\C(\Gcal^I)$ is naturally isomorphic to $\C(\Gcal)^I$, we can view
objects of $\D(\Gcal^I)$ as $I$-shaped diagrams in the category $\C(\Gcal)$ of
cochain complexes. The \newterm{standard derivator} of $\Gcal$ is a $2$-functor
$\mathfrak{D}_\Gcal: \Cat^{\mathrm{op}} \rightarrow \CAT$ satisfying several
properties. First, for any small category $I$, the image $\mathfrak{D}_\Gcal(I)$
is the triangulated category $\D(\Gcal^I)$. In particular, if $\star$ denotes
the category with a single object and a single morphism, we see that
$\mathfrak{D}_\Gcal(\star)$ recovers the derived category $\D(\Gcal)$. Another
property is that given any morphism $u: I \rightarrow J$ in $\Cat$, the induced
functor $\mathfrak{D}_{\Gcal}(u): \mathfrak{D}_{\Gcal}(J) \rightarrow
\mathfrak{D}_{\Gcal}(I)$ is triangulated and it admits both a left and a
right adjoint which are called the \newterm{left and right Kan homotopy
extensions} along $u$. For the full definition of an abstract stable derivator,
we refer the reader e.g. to \cite[Definition~5.9, Definition~5.11]{St2}.

For any small category $I$ and any object $i \in I$, let $\iota_i: \star
\rightarrow I$ denote the unique functor which maps the only object of $\star$
to $i$. The collection of functors $\mathfrak{D}_\Gcal(\iota_i):
\mathfrak{D}_\Gcal(I) \rightarrow \mathfrak{D}_\Gcal(\star)$ induce a functor
$\mathrm{diag}_I: \D(\Gcal^I) = \mathfrak{D}_\Gcal(I) \rightarrow \mathfrak{D}_\Gcal(\star)^I
= \D(\Gcal)^I$ called the \newterm{diagram functor}. It is essentially the
reason why the theory of derivators exists that the diagram functor is usually
\emph{not} an equivalence. We call objects of $\mathfrak{D}_\Gcal(I)$ the
\newterm{coherent diagrams} of shape $I$ in $\D(\Gcal)$, in contrast with
objects of the diagram category $\mathfrak{D}_\Gcal(\star)^I$ which are
sometimes called \newterm{incoherent diagrams}. It is convenient to denote for
any coherent diagram $\Xcal \in \mathfrak{D}_\Gcal(I)$ by
$\Xcal_i:=\mathrm{diag}_I(\Xcal)(i)$ the $i$-th
coordinate of the incoherent diagram $\mathrm{diag}_I(\Xcal)$.

For any small category $I$, denote the unique functor $I \rightarrow \star$ by
$\pi_I$. The left Kan extension along $\pi_I$ has a special name --- it is the
\newterm{homotopy colimit} functor $\hocolim_I: \D(\Gcal^I)=\mathfrak{D}_\Gcal(I)
\rightarrow \mathfrak{D}_\Gcal(\star) = \D(\Gcal)$, and it is equivalent to the
left derived functor of the colimit functor $\Gcal^I\to \Gcal$. In particular, if $I$ is a directed category, the
associated homotopy colimit functor $\hocolim_I$ is computed on a diagram $\Xcal
\in \D(\Gcal^I)$ simply by computing the direct limit $\varinjlim_I(\Xcal)$ of
the diagram $\Xcal \in \C(\Gcal^I)=\C(\Gcal)^I$ inside the Grothendieck category
$\C(\Gcal)$.

There is also a notion of a morphism and equivalence between derivators, we
refer the reader to \cite[\S5]{St2} and \cite{MG}. For our purposes, it will be
enough to say that if $\Gcal,\Ecal$ are two Grothendieck categories, then a
\newterm{morphism} of derivators $\eta: \mathfrak{D}_\Gcal \rightarrow
\mathfrak{D}_{\Ecal}$ induces functors $\eta^I: \mathfrak{D}_\Gcal(I)
\rightarrow \mathfrak{D}_{\Ecal}(I)$ such that for each morphism $u: I
\rightarrow J$ the following square commutes (up to natural equivalence):
\begin{equation}\label{square-derivator}
\begin{tikzcd}
	\mathfrak{D}_\Gcal(J)  \arrow{d}[swap]{\eta^J}\arrow{r}{\mathfrak{D}_\Gcal(u)} & \mathfrak{D}_\Gcal(I) \arrow{d}[swap]{\eta^I}\\
	\mathfrak{D}_{\Ecal}(J) \arrow{r}{\mathfrak{D}_{\Ecal}(u)} & \mathfrak{D}_{\Ecal}(I)
\end{tikzcd}
\end{equation}

The morphism of derivators $\eta$ is an \newterm{equivalence} if all the
functors $\eta^I$ are equivalences. If this is the case, then $\eta$ is an
honest equivalence in a suitable category of derivators
\cite[Proposition~2.11]{MG}, and all the equivalences $\eta^I$ are triangle
equivalences \cite[Proposition~5.12]{St2}. Furthermore, if $\eta$ is an
equivalence then one can check by passing to adjoint functors that $\eta$ is
also compatible with left and right Kan extensions along any morphism $u$ in
$\Cat$. In particular, we get the commutative square for any $I \in \Cat$:
\begin{equation}\label{square-derivator-hocolim}
\begin{tikzcd}
	\mathfrak{D}_\Gcal(I)  \arrow{d}{\cong}[swap]{\eta^I}\arrow{r}{\hocolim_I} & \mathfrak{D}_\Gcal(\star) \arrow{d}{\cong}[swap]{\eta^\star}\\
	\mathfrak{D}_{\Ecal}(I) \arrow{r}{\hocolim_I} & \mathfrak{D}_{\Ecal}(\star)
\end{tikzcd}
\end{equation}

Note that since cohomology is computed coordinate-wise, an object $\Xcal$ of the
bounded derived category $\D^b(\Gcal^I)$ is an $I$-shaped diagram in $\C(\Gcal)$
such that the cohomologies of the coordinates $\Xcal_i$ are uniformly bounded,
that is, there are integers $l<m$ such that $H^j(\Xcal_i) = 0$ for all $j<l$ or
$j>m$ and all $i \in I$. By the exactness of direct limits in $\C(\Gcal)$, we
see that for any small directed category $I$ the homotopy colimit functor
restricts to a functor $\hocolim_I: \D^b(\Gcal^I) \to \D^b(\Gcal)$. We say that
an equivalence of standard derivators $\eta: \mathfrak{D}_\Gcal \rightarrow
\mathfrak{D}_{\Ecal}$ is \newterm{bounded} if for any small category $I$ the
triangle equivalence $\eta^I$ restricts to a triangle equivalence $\eta^I:
\D^b(\Gcal^I) \rightarrow \D^b(\Ecal^I)$. If $I$ is directed, the above
commutative square restricts to another one:
\[ \begin{tikzcd}
	\D^b(\Gcal^I)  \arrow{d}{\cong}[swap]{\eta^I}\arrow{r}{\hocolim_I} &
		\D^b(\Gcal) \arrow{d}{\cong}[swap]{\eta^\star}\\
	\D^b(\Ecal^I) \arrow{r}{\hocolim_I} & \D^b(\Ecal)
\end{tikzcd}\]

In our context, equivalences of standard derivators will appear in the form of
enhancements of (unbounded) realization functors. If $\Tbb$ is a $t$-structure
in $\D(\Gcal)$ with heart $\Hcal$ satisfying certain assumptions, Virili
constructs in \cite[Theorem B, \S6]{Vi} a morphism $\mathfrak{real}_\Tbb:
\mathfrak{D}_\Hcal \rightarrow \mathfrak{D}_\Gcal$ between standard derivators
such that the functor $\mathfrak{real}_\Tbb^\star: \D(\Hcal) \rightarrow
\D(\Gcal)$ is triangulated and restricts to a realization functor $\D^b(\Hcal)
\rightarrow \D(\Gcal)$. We will discuss the cases when this occurs in
(co)tilting theory in \cref{sec:silting}.

\subsection{Intermediate and standard $t$-structures in
	$\D(\Gcal)$}\label{prelim:intermediate}

Let $\Gcal$ be an abelian category. For any integer $n \in \mathbb{Z}$, there is
a $t$-structure $(\D^{\leq n},\D^{> n})$, where $\D^{\leq n} = \{X \in \D(\Gcal)
\mid H^i(X) = 0 ~\forall i> n\}$ and $\D^{> n} = \{X \in \D(\Gcal) \mid H^i(X) =
0 ~\forall i\leq n\},$ called the ($n$-th shift of the) \newterm{standard
$t$-structure}. The left truncation functor $\tau_{\D^{\leq n}}$ is induced by
the \newterm{soft truncation} of complexes and denoted simply by $\tau^{\leq
n}$, similarly the right truncation is the soft truncation functor $\tau^{> n}$.

A $t$-structure $\Tbb = (\Ucal,\Vcal)$ in $\D(\Gcal)$ is \newterm{intermediate}
if there are integers $l<m$ such that $\D^{\leq l} \subseteq \Ucal \subseteq
\D^{\leq m}$, or equivalently, $\D^{> m} \subseteq \Vcal \subseteq \D^{> l}$. It
is easy to see that the intermediacy of the $t$-structure $\Tbb$ yields that the
realization functor $\real^b_{\Tbb}: \D^b(\Hcal_{\Tbb}) \rightarrow \D(\Gcal)$
corestricts to a functor $\real^b_{\Tbb}: \D^b(\Hcal_{\Tbb}) \rightarrow
\D^b(\Gcal)$ between the bounded derived categories.

\section{Krause's recollement for locally coherent Grothendieck categories}

As mentioned in the Introduction, Krause proved the following theorem:
\begin{thm}[{\cite[Corollary~4.3]{K05}}]
	Let $\Gcal$ be a locally noetherian Grothendieck category such that
	$\D(\Gcal)$ is compactly generated. Then there is a recollement
	\[\smallrecoll{\Kac(\inj \Gcal)}{\K(\inj \Gcal)}{\D(\Gcal)}.\]
\end{thm}

Our goal in this section is to prove the same result for $\Gcal$ locally
coherent, rather than locally noetherian.
%Let $\Gcal$ be a locally coherent Grothendieck category such that $\D(\Gcal)$ is
%compactly generated. Our goal in this section is to extend \cite[Corollary
%4.3]{K05} from the locally noetherian to the locally coherent case, that is, to
%show that Krause's recollement exists for $\Gcal$.
Such a result was already established by Šťovíček \cite[Theorem 7.7]{St}, but under the additional
assumption that $\Gcal$ admits a set of finitely presented generators of finite
projective dimension (see \cite[Hypothesis 7.1]{St}).  This assumption implies
that $\D(\Gcal)$ is compactly generated, but it is strictly stronger: we
demonstrate it with an example, which is a Happel-Reiten-Smal\o~ tilt in the derived
category of a commutative noetherian ring, in \cref{example:hyp7.1}. Our
approach here is closer to the original one of Krause, but relies on some of the
results of Šťovíček \cite[Section 6]{St} (these do not depend on the
aforementioned \cite[Hypothesis 7.1]{St}).

Recall that $\fp \Gcal$ denotes the subcategory of all finitely presented
objects of $\Gcal$, an exact abelian subcategory in case $\Gcal$ is locally
coherent. Our starting point is the following result of Šťovíček.

\begin{thm}\cite[Corollary 6.13]{St}\label{theorem:stovicek}
	Let $\Gcal$ be a locally coherent Grothendieck category. Then $\K(\inj
	\Gcal)$ is compactly generated and the functor assigning to an object of
	$\D^b(\fp \Gcal)$ its injective resolution induces an equivalence
	$\D^b(\fp \Gcal) \cong  \K(\inj \Gcal)^c$.
\end{thm}

\begin{cor}\label{corollary:rightadjoint}
	Let $\Gcal$ be a locally coherent Grothendieck category. Then the
	functor $Q: \K(\inj \Gcal) \rightarrow \D(\Gcal)$ admits a right adjoint
	$Q_r$.

	Furthermore, the equivalence $\D^b(\fp \Gcal) \cong \K(\inj \Gcal)^c$
	of \cref{theorem:stovicek} is induced by the restrictions of the adjoint
	functors $Q_r$ and $Q$.
\end{cor}

\begin{proof}
	 By \cref{theorem:stovicek}, $\K(\inj \Gcal)$ is compactly generated.
	 Since $\Gcal$ has exact coproducts, the functor $Q$ preserves
	 coproducts, and so \cite[Theorem 4.1]{Nee96} applies and produces the
	 desired right adjoint.

	It follows directly from the adjunction that for any $X \in \D(\Gcal)$,
	$Q_r(X)$ is homotopy equivalent to a dg-injective resolution of $X$ (which
	exists by \cite{Ser}). By \cref{theorem:stovicek} we have that $Q_r(X)$
	restricts to the equivalence $\D^b(\fp \Gcal) \cong  \K(\inj \Gcal)^c$ with
	the inverse equivalence being the restriction of $Q$ to $ \K(\inj \Gcal)^c$.
\end{proof}

\subsection{Compact objects of $\D(\Gcal)$ and the (small) singularity category}

The main obstacle in extending Krause's proof to the locally coherent case is
showing that any compact object of $\D(\Gcal)$ belongs to $\D^b(\fp \Gcal)$, and
therefore representes a compact object also in $\K(\inj \Gcal)$ via $Q_r$; the
proof in the locally noetherian case \cite[Lemma 4.1]{K05} does not generalize
directly.

Following Gillespie \cite{Gill}, an object $M$ of a Grothendieck category
$\Gcal$ is said to be \emph{of type $\FPinf$} if the functor $\Ext_\Gcal^i(M,-)$
naturally preserves direct limits for all $i \geq 0$. It will be convenient for
our purposes to extend this notion to any object of the bounded derived
category.

\begin{dfn}\label{def:fpinf}
	Let $\Gcal$ be a Grothendieck category. An object $X \in \D^b(\Gcal)$ is
	\newterm{of type $\FPinf$} if for any direct system $(M_i \mid i \in I)$ in
	$\Gcal$ and any $n \in \Zbb$ the natural map
	\[\varinjlim_{i \in I}\Hom_{\D^b(\Gcal)}(X,M_i[n]) \rightarrow
		\Hom_{\D^b(\Gcal)}(X,\varinjlim_{i \in I} M_i[n])\]
	is an isomorphism.
\end{dfn}

Not very surprisingly, \cref{def:fpinf} admits a somewhat more internal
characterization using homotopy colimits of bounded directed coherent diagrams,
which in turn provides a ``bounded'' version of the following notion from the
theory of stable derivators.

\begin{dfn}[{\cite[Definition~5.1]{SSV}}]
	Given a directed small category $I$, $X \in \D(\Gcal)$, and a coherent
	diagram of shape $I$, $\Ycal \in
	\D(\Gcal^I)$, there is a natural map (see \cite[Definition~6.5]{St2})
	\[\varinjlim_{i \in I}\Hom_{\D(\Gcal)}(X,\Ycal_i) \rightarrow
		\Hom_{\D(\Gcal)}(X,\hocolim_I \Ycal).\]
	An object $X \in \D(\Gcal)$ is called \newterm{homotopically finitely
	presented} if the map above is an isomorphism for any choice of $I$ and
	$\Ycal$.
\end{dfn}

\begin{lemma}\label{lemma:hcompact-1}
	Let $\Gcal$ be a Grothendieck category. An object $X \in \D^b(\Gcal)$ is of
	type $\FPinf$ if and only if for any directed small category $I$ and any
	coherent diagram $\Ycal \in \D^b(\Gcal^I)$ the natural map
	\[\varinjlim_{i \in I}\Hom_{\D^b(\Gcal)}(X,\Ycal_i) \rightarrow
		\Hom_{\D^b(\Gcal)}(X,\hocolim_I \Ycal)\]
	is an isomorphism.
\end{lemma}

\begin{proof}
	Since $\Ycal$ belongs to $\D^b(\Gcal^I)$, the coherent diagram $\Ycal$ is represented
	by a direct system $(Y_i \mid i \in I)$ in $\C(\Gcal)$ such that the
	cohomology of the complexes $Y_i$ is uniformly bounded. Therefore, there is
	$n \in \Zbb$ and $k \geq 0$ such that for all $i \in I$, the cohomology of
	$Y_i$ vanishes outside of degrees $n,\ldots,n+k$.  If $k=0$, by applying the
	soft truncation we may assume that $\Ycal$ is such that $(Y_i \mid i \in I)$
	is a direct system of stalk complexes in degree $n$, and therefore the
	required isomorphism is provided by the definition of an object of type
	$\FPinf$. The general case follows by induction on $k>0$.  Indeed, applying
	$\hocolim_I$ to the soft truncation triangle of $\Ycal$ in $\D^b(\Gcal^I)$
	we obtain the triangle
	\[\hocolim_I\tau^{\leq {n}}\Ycal \rightarrow \hocolim_I\Ycal \rightarrow
		\hocolim_I\tau^{> {n}}\Ycal \xrightarrow{+}\]
	in $\D^b(\Gcal)$. Notice that soft truncations commute naturally with the
	component functors $(-)_i$, and we have triangles in $\D^b(\Gcal)$
	\[\tau^{\leq {n}}\Ycal_i \rightarrow \Ycal_i \rightarrow \tau^{>
		{n}}\Ycal_i \xrightarrow{+}.\]
	
	There is the following commutative diagram, in which the horizontal maps
	are induced by the two triangles above and the vertical ones are the natural
	maps (we write $(A,B):=\Hom_{\D(\Gcal)}(A,B)$, to lighten the
	notation):
	\begin{center}
		\hspace*{-0.6pc}
	\adjustbox{scale = 0.9}{%
		\begin{tikzcd}[row sep = small, column sep = tiny]
	\varinjlim_{i\in I}(X,\tau^{> {n}}\Ycal_i[-1]) \arrow{d}\arrow{r} & \varinjlim_{i\in I}(X,\tau^{\leq{n}}\Ycal_i) \arrow{d}\arrow{r} & \varinjlim_{i\in I}(X,\Ycal_i) \arrow{d}\arrow{r} & \varinjlim_{i\in I}(X,\tau^{> {n}}\Ycal_i)  \arrow{d}\arrow{r} & \varinjlim_{i\in I}(X,\tau^{\leq{n}}\Ycal_i[1]) \arrow{d} \\
	(X,\hocolim_I \tau^{> {n}}\Ycal[-1]) \arrow{r} & (X,\hocolim_I \tau^{\leq n}\Ycal) \arrow{r} & (X,\hocolim_I \Ycal) \arrow{r} & (X,\hocolim_I \tau^{> {n}}\Ycal) \arrow{r} & (X,\hocolim_I \tau^{\leq {n}}\Ycal[1])
\end{tikzcd}
		}
\end{center}

	 Then the induction step follows directly by Five lemma, as both the
	 coherent diagrams $\tau^{>n} \Ycal$ and $\tau^{\leq n}\Ycal$ are subject to
	 the induction hypothesis for $k-1$.
	%$\varinjlim_{i\in
	%I}\Hom_{\D^b(\Gcal)}(X,(\mathrm{B}{}_i))$ and
	%$\Hom_{\D^b(\Gcal)}(X,(\mathrm{A}))$ and the natural
	%transformation $\varinjlim_{i \in I}\Hom_{\D^b(\Gcal)}(X,(-)_i) \rightarrow
	%\Hom_{\D^b(\Gcal)}(X,\hocolim_I (-))$ between functors $\D(\Gcal^I)\to
	%\mathsf{Ab}$.
	%{\color{blue} taking values in $\D(\Gcal^I)$.}
\end{proof}

\begin{lemma}\label{lemma:hcompact-2}
	Let $\Gcal$ be a Grothendieck category. The objects of type $\FPinf$ of $\D^b(\Gcal)$ form a thick subcategory of $\D^b(\Gcal)^c$.
\end{lemma}

\begin{proof}
	By exactness of coproducts in $\Gcal$, the coproducts in $\D^b(\Gcal)$ are
	precisely the coproducts of collections of objects with uniformly bounded
	cohomology, computed in $\D(\Gcal)$. Therefore, any coproduct in
	$\D^b(\Gcal)$ can be realized as a directed homotopy colimit of a suitable
	diagram of $\D^b(\Gcal^I)$ whose components are finite subcoproducts. In
	this way \cref{lemma:hcompact-1} shows that any object of type $\FPinf$ in
	$\D^b(\Gcal)$ is compact in $\D^b(\Gcal)$. The fact that objects of type
	$\FPinf$ form a thick subcategory follows from the Five lemma
	similarly as in the proof of \cref{lemma:hcompact-1}, the closure under
	retracts is clear.
\end{proof}

\begin{lemma}\label{lemma:fp-equal-hcompact}
	Let $\Gcal$ be a locally coherent Grothendieck category. An object $X \in
	\D^b(\Gcal)$ is of type $\FPinf$ if and only if $X \in \D^b(\fp \Gcal)$.
\end{lemma}

\begin{proof}
	An object $F \in \fp \Gcal$ is of type $\FPinf$ as an object in
	$\D^b(\Gcal)$, see \cite[Theorem 3.21]{Gill}. By \cref{lemma:hcompact-2},
	any object in the thick closure of $\fp \Gcal$ in $\D^b(\Gcal)$ is of type
	$\FPinf$, which shows that $X \in \D^b(\fp \Gcal)$ implies that $X$ is of
	type $\FPinf$.

	For the converse implication, let $X \in \D^b(\Gcal)$ be of type $\FPinf$
	and let $n$ be a maximal integer such that $H^n(X) \neq 0$. For any $M \in
	\Gcal$ the soft truncation yields a natural isomorphism
	$\Hom_{\D^b(\Gcal)}(X,M[-n]) \cong \Hom_{\Gcal}(H^n(X),M)$. Since $X$ is of
	type $\FPinf$, it follows that the functor $\Hom_{\Gcal}(H^n(X),-): \Gcal
	\to \Mod \mathbb{Z}$ preserves direct limits, and so $H^n(X)$ belongs to
	$\fp \Gcal$. Using the previous paragraph and \cref{lemma:hcompact-2} we
	infer that the soft truncation $\tau^{<n}X$ is of type $\FPinf$. Continuing
	by finite induction we conclude that all cohomologies of $X$ belong to $\fp
	\Gcal$, and so $X \in \D^b(\fp \Gcal)$, see e.g. \cite[Theorem 15.3.1]{KS}.
\end{proof}

\begin{rmk}\label{rem:inclusion-bounded-compact}
	Combining \cref{lemma:fp-equal-hcompact} and \cref{lemma:hcompact-2} we
	obtain the inclusion $\D^b(\fp {\Gcal}) \subseteq \D^b(\Gcal)^c$. We do not
	know whether the converse inclusion holds in general for a locally
	coherent Grothendieck category such that $\D(\Gcal)$ is compactly generated.
	However, in \cref{sec:restrictable}, we will show that these two
	categories coincide in case $\Gcal$ is the heart of an intermediate
	cotilting $t$-structure over a commutative noetherian ring.
\end{rmk}

\begin{prop}\label{prop:compacts-are-bounded}
	Let $\Gcal$ be a locally coherent Grothendieck category. There is an
	inclusion $\D(\Gcal)^c \subseteq \D^b(\fp \Gcal)$.
\end{prop}

\begin{proof}
	Let $C$ be a compact object of $\D(\Gcal)$. For each $n\in\Zbb$ there is
	a natural map $C \to E(H^n(C))[-n]$ in $\D(\Gcal)$ to a shift of the injective
	envelope of $H^n(C)$. This induces a morphism $C\to
	\prod_{n\in\Zbb}E(H^n(C))[-n]$. Products in $\D(\Gcal)$ are computed
	as component-wise products of dg-injective resolutions; so in this case,
	the component-wise product of the $E(H^n(C))[-n]$.
	In this particular case, it coincides with the component-wise
	coproduct. This is also the
	coproduct in $\D(\Gcal)$, since $\Gcal$ has exact coproducts. Therefore we obtain a morphism $C\to
	\coprod_{n\in\Zbb} E(H^n(C))[-n]$ in $\D(\Gcal)$.
	%Let $C$ be a complex over $\Gcal$, which is compact in $\D(\Gcal)$.
	%For each $n \in \Zbb$, there is a
	%natural map $C \rightarrow E(H^n(C))$ in $\K(\Gcal)$ to the injective
	%envelope of the $H^n(C)$. Consider the induced map $C \rightarrow
	%\prod_{n \in \Zbb}E(H^n(C))[-n] \cong \coprod_{n \in
	%\Zbb}E(H^n(C))[-n]$ in $\K(\Gcal)$. This represents a morphism
	%$C\to \coprod_{n\in\Zbb}E(H^n(C))[-n])$ in $\D(\Gcal)$, where the
	%coproduct is also taken in $\D(\Gcal)$.
	By compactness of $C$, this map factors through a finite subcoproduct. It
	follows that $C$ has finitely many non-zero cohomologies, i.e.
	$C\in\D^b(\Gcal)$.

	By \cite[Corollary~6.10]{SSV}, $C$ is homotopically finitely presented in
	$\D(\Gcal)$. In particular, $C$ is of type $\FPinf$ in $\D^b(\Gcal)$.
	Therefore, $C \in \D^b(\fp \Gcal)$ by \cref{lemma:fp-equal-hcompact}.
\end{proof}

\begin{rmk}\label{rmk:dsing}
	Let $\Gcal$ be a locally coherent Grothendieck category.
	\cref{prop:compacts-are-bounded} shows that $\D(\Gcal)^c$ is a thick
	subcategory of $\D^b(\fp \Gcal)$, and therefore we can form the Verdier
	quotient $\D^{\sing}(\Gcal) = \D^b(\fp \Gcal)/\D(\Gcal)^c$. Following the
	locally noetherian case \cite{K05}, we call $\D^{\sing}(\Gcal)$ the
	\newterm{(small) singularity category} of $\Gcal$.
\end{rmk}

\subsection{The recollement}

In order to prove that $\Gcal$ admits Krause's recollement we need to construct a
left adjoint to the localization $Q\colon\K(\inj\Gcal)\to \D(\Gcal)$.

\begin{lemma}\label{lemma:adjointcompact}
	Let $\Gcal$ be a locally coherent Grothendieck category. For any $C \in
	\D(\Gcal)^c$ and any $Y \in \K(\inj \Gcal)$, there is a natural
	isomorphism 
	\[\Hom_{\D(\Gcal)}(C,QY) \cong \Hom_{\K(\inj \Gcal)}(Q_r C,Y).\]
\end{lemma}

\begin{proof}
	Consider the natural transformation 
	\[\eta_{C,Y}:  \Hom_{\K(\inj \Gcal)}(Q_r C,Y) \rightarrow
		\Hom_{\D(\Gcal)}(QQ_r C,QY)\]
	induced by $Q$. By \cref{corollary:rightadjoint}, the functors $Q_r$ and
	$Q$ induce an equivalence $\D^b(\fp \Gcal) \cong \K(\inj \Gcal)^c$. We
	see that $QQ_r C$ is naturally isomorphic to $C$ and also, in view of
	\cref{prop:compacts-are-bounded},  that $\eta_{C,Y}$ is an isomorphism
	whenever $Y \in \K(\inj \Gcal)^c$. Consider the subcategory $\Kcal$ of
	$\K(\inj \Gcal)$ consisting of all objects $Y$ such that $\eta_{C,Y}$ is
	an isomorphism for all $C \in \D(\Gcal)^c$. A standard argument shows
	that $\Kcal$ is a triangulated subcategory. Since $C$ is compact in
	$\D(\Gcal)$ and $Q_r C$ is compact in $\K(\inj \Gcal)$, the subcategory
	$\Kcal$ is closed under coproducts. Then $\Kcal$ is a localizing
	subcategory of $\K(\inj \Gcal)$ containing all compact objects, and
	therefore $\Kcal = \K(\inj \Gcal)$ by \cref{theorem:stovicek}.
\end{proof}

\begin{lemma}\label{lemma:leftadjoint}
	Let $\Gcal$ be a locally coherent Grothendieck category such that
	$\D(\Gcal)$ is compactly generated. Then the functor $Q: \K(\inj \Gcal)
	\rightarrow \D(\Gcal)$ admits a left adjoint $Q_l$.
\end{lemma}

\begin{proof}
	Let $\Lcal$ be the localizing subcategory of $\K(\inj \Gcal)$ generated by
	$Q_r(\D(\Gcal)^c)$. Then $\Lcal$ is a compactly generated triangulated
	category, and the restriction $Q_{\restriction \Lcal}: \Lcal \rightarrow
	\D(\Gcal)$ is a functor between compactly generated triangulated categories
	that preserves coproducts and by \cref{corollary:rightadjoint} restricts
	further to an equivalence $\Lcal^c \cong \D(\Gcal)^c$. Then $Q_{\restriction
	\Lcal}$ is an equivalence by \cref{lemma:double-devissage}, and so there is
	an inverse equivalence $P: \D(\Gcal) \toeq \Lcal$. We define $Q_l$ as the
	composition of $P$ and the inclusion $\iota: \Lcal \hookrightarrow{} \K(\inj
	\Gcal)$. 

	The inclusion $\iota$ of $\Lcal$ into $\K(\inj \Gcal)$ has a right adjoint
	$\tau: \K(\inj \Gcal) \rightarrow \Lcal$, see e.g. \cite[Theorem
	4.1]{Nee96}. It follows that $Q_l = \iota \circ P $ has a right adjoint $Q
	\circ \tau$. It remains to show that $Q \circ \tau$ is naturally equivalent
	to $Q$. Applying $Q$ to the counit transformation $\iota \circ \tau
	\rightarrow \mbox{id}_{\K(\inj \Gcal)}$ we see that it is enough to show
	that any object of $\Lcal\Perp{0}$ is sent to zero by $Q$, i.e.
	$\Lcal\Perp{0} \subseteq \Kac(\inj \Gcal)$. If $Y \in \Lcal\Perp{0}$ then
	$\Hom_{\K(\inj \Gcal)}(Q_r C,Y) = 0$ for all $C \in \D(\Gcal)^c$. By
	\cref{lemma:adjointcompact}, this implies $\Hom_{\D(\Gcal)}(C,QY) = 0$ for
	all $C \in \D(\Gcal)^c$, and since $\D(\Gcal)$ is compactly generated, we
	have $QY = 0$, as desired.
\end{proof}

We record the following auxiliary property of the adjoints of $Q$ for later use. 

\begin{lemma}\label{lem-Ql-Qr}
	In the setting of \cref{lemma:leftadjoint} we have an isomorphism $Q_r C
	\cong Q_l C$ for all $C \in \D(\Gcal)^c$.
\end{lemma}

\begin{proof}
	By \cref{lemma:leftadjoint} and \cref{lemma:adjointcompact}, there are
	natural isomorphisms for all $Y \in \K(\inj \Gcal)$
	\[\Hom_{\K(\inj \Gcal)}(Q_l C, Y) \cong \Hom_{\D(\Gcal)}(C, QY) \cong
		\Hom_{\K(\inj \Gcal)}(Q_r C, Y).\]
	The isomorphism $Q_r C \cong Q_l C$ thus follows from the Yoneda lemma.
\end{proof}

\begin{thm}\label{theorem:recollement}
	Let $\Gcal$ be a locally coherent Grothendieck category such that
	$\D(\Gcal)$ is compactly generated. Then there is a recollement:
	\[\begin{tikzcd}[column sep=huge]
		\Kac(\inj \Gcal) \arrow{r}{i_*} &
		\K(\inj \Gcal) \arrow[shift left,  bend left] {l}       {i^!}
			       \arrow[shift right, bend right]{l}[above]{i^*}
			       \arrow{r}{Q} &
		\D(\Gcal)\quad \arrow[shift left,  bend left] {l}       {Q_r}
		          \arrow[shift right, bend right]{l}[above]{Q_l}
	\end{tikzcd}\]
\end{thm}

\begin{proof}
	Recall that the functor $Q$ is a Verdier localization functor whose
	kernel is the full subcategory $\Kac(\inj \Gcal)$. By a standard
	argument (see \cite[Lemma 3.2]{K05}), it is enough to establish that $Q$
	admits both left and right adjoint functors, which we showed in
	\cref{corollary:rightadjoint} and \cref{lemma:leftadjoint}.
\end{proof}

\begin{cor}\label{cor:Kac}
	In the setting of \cref{theorem:recollement}, the category $\Kac(\inj
	\Gcal)$ is compactly generated and the subcategory of compact objects
	$\Kac(\inj \Gcal)^c$ is equivalent up to retracts to the singularity
	category $\D^{\sing}(\Gcal)$ of $\Gcal$.
\end{cor}

\begin{proof}
	This follows directly from \cite[Theorem 2.1]{Nee92} applied to the
	situation of \cref{theorem:recollement}.
\end{proof}

\begin{cor}[{cf. \cite[Corollary 4.4]{K05}}]
	Let $\Gcal$ be a locally coherent Grothendieck category such that
	$\D(\Gcal)$ is compactly generated. Then any product of acyclic complexes of
	injective objects is acyclic.
\end{cor}

\begin{rmk}\label{rem:stablederived}
	In the locally noetherian situation \cite{K05}, the category $\Kac(\inj
	\Gcal)$ is called the \newterm{stable derived category} of $\Gcal$ and
	denoted by $\DS(\Gcal)$, while other sources \cite{Bec14}, \cite{St} call
	it the \newterm{(large) singularity category} of $\Gcal$. In the latter two
	citations, it is shown that $\DS(\Gcal)$ is a homotopy category of
	$\C(\Gcal)$ endowed with a suitable abelian model structure. It is also
	explained in \cite[\S7]{St} that $\DS(\Gcal)$ naturally identifies with the
	subcategory of all acyclic complexes of the coderived category
	$\D^\co(\Gcal)$ via the equivalence $\K(\inj{\Gcal}) \cong \D^{\co}(\Gcal)$,
	and the same equivalence identifies the recollement of
	\cref{theorem:recollement} with the recollement of the form
	\[\begin{tikzcd}[column sep=huge]
		\DS(\Gcal) \arrow{r}{\subseteq} &
		\D^\co(\Gcal) \arrow[shift left,  bend left] {l}       {}
			       \arrow[shift right, bend right]{l}[above]{}
			       \arrow{r}{Q} &
		\D(\Gcal)\quad \arrow[shift left,  bend left] {l}       {}
		          \arrow[shift right, bend right]{l}[above]{}
	\end{tikzcd}\]
\end{rmk}

\subsection{The singularity category of a locally coherent Grothendieck category}

The next goal is to interpret the vanishing of the singularity category
$\D^\sing(\Gcal)$ of \cref{rmk:dsing} in terms of homological dimension of objects of $\Gcal$. For
this, we need to impose a relatively mild condition on $\Gcal$. Following Roos
\cite{Roo}, a Grothendieck category $\Gcal$ is $\Abfs d$ for a non-negative
integer $d$ if for any set $I$, the product functor $\prod_I: \Gcal^I
\rightarrow \Gcal$ has cohomological dimension at most $d$, we refer the reader
to \cite{HX} for further details. In particular, $\Gcal$ satisfies $\Abfs 0$ if
and only if the products are exact in $\Gcal$. Recall that the product $\prod_{i
\in I}M_i$ in $\D(\Gcal)$ is computed as the product $\prod_{i \in I}E_i$ in
$\C(\Gcal)$ for any choice of injective resolutions $E_i$ of $M_i$. Therefore,
$\Gcal$ satisfies $\Abfs d$ if and only if the (component-wise) product
$\prod_{i \in I}E_i$ belongs to $\D^{\leq d}$ whenever $E_i$ are complexes of
injective objects of $\Gcal$ concentrated in non-negative degrees with the only
non-vanishing cohomology in degree zero.

\begin{lemma}\label{lem:cosmashing}
	Let $\Gcal$ be Grothendieck category which is $\Abfs d$ for some $d \geq 0$.
	Then $\prod_{i \in I}X_i \in \D^{\leq d}$ for any collection of objects $X_i
	\in \D^{\leq 0}$ (in other words, the standard $t$-structure $(\D^{\leq
	0},\D^{>0})$ in $\D(\Gcal)$ is \newterm{$d$-cosmashing}, see
	\cite[Definition~5.4]{Vi}).
\end{lemma}

\begin{proof} Let $E_i$ be a dg-injective replacement of $X_i$ for any $i \in I$
	\cite{Ser}.  Now we need to show that the (component-wise) product $\prod_{i
	\in I}E_i$ is in $\D^{\leq d}$. First, let us assume that there is $k \leq
	0$ such that all components of $E_i$ in degrees below $k$ are zero. If $k =
	0$ then $E_i$ are injective resolutions of objects of $\Gcal$ and so
	$\prod_{i \in I}E_i \in \D^{\leq d}$ by the definition of the $\Abfs d$
	property. The case of $k<0$ is proved using induction and brutal
	truncations, using the fact that brutal truncations commute with
	component-wise products. Finally, if there is no such $k$ we argue using the
	following isomorphism: $\prod_{i \in I}E_i \cong \varinjlim_{n<0}
	\sigma^{>n} \prod_{i \in I}E_i \cong \varinjlim_{n<0} \prod_{i \in
	I}\sigma^{>n}  E_i$, where $\sigma^{>n}$ is the brutal truncation to degrees
	above $n$. Then $\prod_{i \in I}\sigma^{>n}  E_i \in \D^{\leq d}$ by the
	previous case, and the directed colimit also stays in $\D^{\leq d}$ by
exactness.  \end{proof} \begin{lemma}\emph{(cf. \cite[\S1.6]{KrDer},
	\cite[Theorem 1.3]{HX})}\label{lem:product-coproduct} Let $\Gcal$ be a
	Grothendieck category satisfying $\Abfs d$ for some $d \geq 0$. Then for any
	collection $M_n \in \Gcal$ of objects indexed by $n \in \mathbb{Z}$ we have
	an isomorphism $\coprod_{n \in \mathbb{Z}}M_n[n] \cong \prod_{n \in
	\mathbb{Z}}M_n[n]$ in $\D(\Gcal)$ (the product is computed in $\D(\Gcal)$).
\end{lemma} \begin{proof} Consider the canonical morphism $\eta: \coprod_{n \in
	\mathbb{Z}}M_n[n] \rightarrow \prod_{n \in \mathbb{Z}}M_n[n]$ in $\D(\Gcal)$
	and let us show that $\eta$ is a an isomorphism. Let $l \in \mathbb{Z}$ and
	let us compute $H^l(\eta)$. Since coproducts are exact in $\Gcal$, the
	coproduct is equivalently computed in $\C(\Gcal)$ and $H^l(\coprod_{n \in
	\mathbb{Z}}M_n[n])$ is clearly just $M_n$. On the other hand, the product
	$\prod_{n \in \mathbb{Z}}M_n[n]$ isomorphic in $\D(\Gcal)$ to the product
	$\prod_{n \in \mathbb{Z}}E_n[n]$ computed in $\C(\Gcal)$, where $E_n$ is an
	injective resolution of $M_n$ for each $n$. Consider the decomposition
	$\prod_{n \in \mathbb{Z}}E_n[n]$ = $\prod_{n > -l + d}E_n[n] \times \prod_{n
	= -l,\ldots,-l+d}E_n[n] \times \prod_{n < -l}E_n[n]$. Clearly, $\prod_{n <
	-l}E_n[n] \in \D^{>l}$. For any $n > -l + d$, we have $E_n[n] \cong M_n[n]
	\in \D^{<l-d}$, and using \cref{lem:cosmashing} we conclude that $\prod_{n >
	-l + d}E_n[n] \in \D^{<l}$. It follows that $H^l(\eta)$ factors as a map
	$H^l(\eta): M_l \rightarrow H^l(\prod_{n =
	-l,\ldots,-l+d}E_n[n])=\prod_{n=-l,\dots,-l+d}
	H^l(E_n[n])=H^l(E_l[l])$, and this is clearly an isomorphism.
\end{proof}
%The following definition is not necessarily standard.
\begin{dfn}\label{dfn-fin-proj}
	An object $X$ of $\D(\Gcal)$ is of \newterm{finite projective dimension} if
	there is $n \in \mathbb{Z}$ such that $\Hom_{\D(\Gcal)}(X,\D^{\leq n})=0$.
\end{dfn}
\begin{rmk}\label{rem:fin-proj}
	If $M \in \Gcal$ and $k > 0$, then $\Hom_{\D(\Gcal)}(M,\D^{\leq -k})=0$
	holds if and only if $\Ext_\Gcal^i(M,N) = 0$ for all $i \geq k$ and $N \in
	\Gcal$; this is proved in large generality in \cite[Lemma 10]{NSZ}.
	Therefore, $M$ viewed as an object of $\D(\Gcal)$ is of finite projective
	dimension in the sense of \cref{dfn-fin-proj} if and only if $M$ has finite
	projective dimension in $\Gcal$ in the usual sense (as used e.g. in
	\cite[Hypothesis~7.1]{St}).

	Note also that essentially the same argument as the one in \cite[Lemma
	10]{NSZ} shows the following: $X \in \D^b(\Gcal)$ is of finite projective
	dimension if and only if $\Hom_{\D(\Gcal)}(X,\Gcal[i])$ for $i \gg 0$.

	%If $X\in\D(\Gcal)$ is of finite projective dimension in the sense of
	%\cref{dfn-fin-proj}, then in particular we have that
	%$\Hom_{\D(\Gcal)}(X,\Gcal[i])=0$ for $i\gg 0$. If $X$ is an object of $\Gcal$,
	%this means precisely that $\Ext_\Gcal^i(X,-)=0$ for $i\gg 0$, i.e. that $X$
	%has finite Yoneda projective dimension; see e.g. \cite[Hypothesis 7.1]{St}. The %converse implication
	%is however not so clear,
	%at least without further assumptions on $\Gcal$. If $\D(\Gcal)$ is
	%\newterm{left-complete} \cite[Definition 6.2]{PV}, that is if any $X \in
	%\D(\Gcal)$ is isomorphic to the homotopy limit of its soft truncations from
	%below, then a standard argument similar to the one in the following
	%paragraph shows that these two definitions are equivalent. It is shown in
	%\cite[Theorem 1.3]{HX} that $\D(\Gcal)$ is left-complete whenever $\Gcal$
	%satisfies $\Abfs d$ for some $d \geq 0$.

	%Even without additional assumptions on $\Gcal$, the two definitions coincide whenever $X$ is a compact object of $\D(\Gcal)$. Indeed, then the condition $\Hom_{\D(\Gcal)}(X,\D^{\leq n})=0$ can be checked just on bounded complexes from $\D^{\leq n}$ by a standard argument using homotopy colimits of hard truncations and the fact that $X$ is compact. Arguing by dimension shifting and finite extensions, $\Hom_{\D(\Gcal)}(X,\D^{\leq n})=0$ is then equivalent to $\Hom_{\D(\Gcal)}(X,\Gcal[-n])=0$.
\end{rmk}
\begin{prop}\label{prop-sing-fp}
	Let $\Gcal$ be a locally coherent Grothendieck category such that
	$\D(\Gcal)$ is compactly generated. Consider the conditions for an object $X \in \D^b(\fp \Gcal)$:
	\begin{enumerate}
		\item[(i)] $X$ is of finite projective dimension,
		\item[(ii)] $X \in \D(\Gcal)^c$.
	\end{enumerate}
	The $(i) \implies (ii)$. If in addition $\Gcal$ satisfies $\Abfs d$ for some $d \geq 0$ then $(i) \iff (ii)$.
\end{prop}

\begin{proof}
	$(i) \Rightarrow (ii):$ Let $m \geq 0$ be such that $X \in D^{\leq m}(\Gcal)$. Because $X$ is of finite projective
	dimension, there is an $n \leq 0$ such that we can use soft
	truncations to obtain for any collection of objects $X_i \in \D(\Gcal), i
	\in I$ a chain of natural isomorphisms $\Hom_{\D(\Gcal)}(X,\coprod_{i \in
	I}X_i) \cong\Hom_{\D(\Gcal)}(X,\tau^{\geq n}\tau^{\leq m}\coprod_{i \in
	I}X_i) \cong \Hom_{\D(\Gcal)}(X,\coprod_{i \in I}\tau^{\geq n}\tau^{\leq
	m}X_i)$. It follows that $X$ is compact in $\D(\Gcal)$ if and only if it is
	compact in $\D^b(\Gcal)$. But $X \in \D^b(\Gcal)^c$ by
	\cref{rem:inclusion-bounded-compact}.

	$(i) \Rightarrow (ii):$ Suppose that there is $X \in \D^b(\fp \Gcal)$ which is
	not of finite projective dimension, and assume towards contradiction that $X$ is compact as an object of $\D(\Gcal)$. In view of \cref{rem:fin-proj}, there are objects $M_n \in \Gcal$
	for all $n \geq 0$ such that $\Ext_\Gcal^n(X,M_n) \neq 0$. By the $\Abfs d$ assumption and \cref{lem:product-coproduct}, there is an isomorphism $\coprod_{n \geq
	0}M_n[n] \cong \prod_{n \geq 0}M_n[n]$ in $\D(\Gcal)$, and so there is a
	morphism $X \rightarrow \coprod_{n \geq 0}M_n[n]$ which does not factor
	through any finite subcoproduct of $\coprod_{n \geq 0}M_n[n]$ in $\D(\Gcal)$
	(cf. \cite[Remark 1.11]{St2}). It follows that $X$ is not compact in
	$\D(\Gcal)$, a contradiction.
\end{proof}
\begin{cor}\label{prop-sing}
	Let $\Gcal$ be a locally coherent Grothendieck category such that
	$\D(\Gcal)$ is compactly generated. Consider the following conditions:
	\begin{enumerate}
		\item[(i)] $\D^\sing(\Gcal) = 0$,
		\item[(ii)] $\DS(\Gcal) = 0$,
		\item[(iii)] any object $F \in \fp \Gcal$ has finite projective dimension. 
	\end{enumerate}
	Then $(i) \iff (ii)$ and $(iii) \implies (i)$. If in addition $\Gcal$ satisfies $\Abfs d$ for some $d \geq 0$ then all the conditions are equivalent.
\end{cor}

\begin{proof}
	The equivalence of $(i)$ and $(ii)$ is clear from \cref{cor:Kac}.

	$(iii) \Rightarrow (i):$ The assumption implies that any object $F \in \D^b(\fp \Gcal)$ is of finite projective dimension, and the conclusion follows by \cref{prop-sing-fp}.

	$(i) \Rightarrow (iii):$ Suppose that there is $F \in \fp \Gcal$ which is
	not of finite projective dimension, then $F \not\in \D(\Gcal)^c$ by \cref{prop-sing-fp} and so it is a non-zero object in $\D^\sing(\Gcal)$.
\end{proof}

We conclude this section by showing that \cref{prop-sing} specializes neatly to
the case of the category of quasicoherent sheaves over a scheme. Following
\cite{Gar}, a quasicompact and quasiseparated scheme $X$ is \newterm{coherent}
if it admits a cover $X = \bigcup_{i \in I}\Spec{R_i}$ by open affine sets such
that $R_i$ is a coherent commutative ring for all $i \in I$. By a standard
argument \cite[Corollary 2.1]{Harris}, this is equivalent to any open affine
subset $\Spec R$ of $X$ being such that the ring $R$ is coherent. By
\cite[Proposition~9.2]{Gar}, $X$ is coherent if and only if the Grothendieck
category $\Qcoh X$ of quasicoherent sheaves is locally coherent. It follows that
$\fp {\Qcoh X} = \coh X$, the category of coherent sheaves, and $\D^b(\fp {\Qcoh
X}) = \D^b(\coh X)$.

The classical notion of a regular noetherian ring admits the following
generalization to coherent rings, here we follow \cite{Ber} and \cite{Glaz}. A
coherent commutative ring $R$ is \newterm{regular} if any finitely generated
ideal has finite projective dimension. By \cite{Glaz}, this is equivalent to any
finitely presented $R$-module being of finite projective dimension. 

It is then natural to call a coherent scheme $X$ \newterm{regular} if it admits
a cover $X = \bigcup_{i \in I}\Spec {R_i}$ where $R_i$ are regular coherent
rings. Since regular coherent rings descent along faithfully flat morphisms
\cite[Theorem 6.2.5]{Glaz}, this is equivalent to any open affine subset $\Spec
R$ of $X$ being such that $R$ is regular coherent.

\begin{rmk}\label{orlov-scheme}
If $X$ is in addition separated, the derived category $\D(\Qcoh X)$ is compactly
generated by \cite[\S3]{BvB}, and the compact objects are up to isomorphism
precisely the perfect complexes, that is, complexes which are locally
quasi-isomorphic to bounded complexes of vector bundles. Therefore, our singularity category $\D^\sing(\Qcoh X)$ takes the familiar form $\D^b(\coh X)/\Per X$, where $\Per X$ is the full subcategory of objects quasi-isomorphic to perfect complexes. In particular, our notion of singularity category recovers Orlov's original definition of a singularity category for separated noetherian schemes \cite{orlo-06}.
\end{rmk}

\begin{prop}\label{cor:scheme}
	Let $X$ be a separated coherent scheme. The following are equivalent:
	\begin{enumerate}
		\item[(i)] $\D^\sing(\Qcoh X) = 0$,
		\item[(ii)] $X$ is regular,
		\item[(iii)] any coherent sheaf has finite projective dimension in $\Qcoh X$.
	\end{enumerate} 
\end{prop}

\begin{proof}
	$(i) \Leftrightarrow (ii)$: If $\mathcal{F} \in \D^b(\coh X)$, we can check whether
	$\mathcal{F} \in \D(\Qcoh X)^c$ locally on an open affine cover, and any
	such restriction becomes a bounded complex of finitely presented modules.
	Therefore, since regularity is also a local property, the task reduces to
	the case of $X$ being an affine scheme. But this case follows directly from
	\cref{prop-sing}, because for an affine scheme $X$ the category $\Qcoh X$ is
	equivalent to a module category, and thus has exact products.

	$(i) \Leftrightarrow (iii):$ Since $\coh X = \fp{\Qcoh X}$, this follows
	from \cref{prop-sing} because $\Qcoh X$ satisfies $\Abfs d$ for some $d \geq
	0$, see \cite[Remark 3.3]{HX}.
\end{proof}

\subsection{Singularity category of a small abelian category}\label{subsec:small} Let $\Acal$ be a (skeletally) small abelian category and $\D^b(\Acal)$ its bounded derived category. Following Roos \cite{Roo69}, we assign to $\Acal$ the category $\widehat{\Acal} := \Lex(\Acal^{op},\Ab)$ of all left exact contravariant additive functors from $\Acal$ to the category of abelian groups. Alternatively, following Crowley-Boevey \cite{CW94}, $\widehat{\Acal}$ can be described as the category of all flat right $\Acal$-modules, when we consider $\Acal$ as a ring with several objects. Then the assignment $\Acal \to \widehat{\Acal}$ induces a bijective correspondence between the equivalence classes of small abelian categories and the equivalence classes of locally coherent Grothendieck categories, \cite[Proposition 2]{Roo69} or \cite[1.4]{CW94}. The converse assignment assigns to a locally coherent Grothendieck category $\Gcal$ the skeletally small abelian category $\fp \Gcal$. In particular, $\fp {\widehat{\Acal}}$ identifies with $\Acal$.

This viewpoint allows us to transport our definition of a singularity category to a large class of small abelian categories. Let $\Acal$ be a small abelian category such that $\D(\widehat{\Acal})$ is compactly generated. Then we define the singularity category of $\Acal$ to be precisely the singularity category of the locally coherent Grothendieck category $\widehat{\Acal}$, that is, we put 
$$\D^\sing(\Acal) := \D^\sing(\widehat{\Acal}) = \D^b(\fp{\widehat{\Acal}})/\D(\widehat{\Acal})^c \cong \D^b(\Acal)/\D(\widehat{\Acal})^c.$$

Assume now that $\Acal$ is a small abelian category which admits a set of generators such that their projective dimension is uniformly bounded by some $d \geq 0$. It follows that the locally coherent Grothendieck category $\widehat{\Acal}$ satisfies \cite[Hypothesis 7.1]{St}. As a consequence, $\D(\widehat{\Acal})$ is compactly generated \cite[Proposition 7.4]{St}. Furthermore, the same assumption ensures that $\widehat{\Acal}$ has a generator of projective dimension $d$, and therefore $\widehat{\Acal}$ satisfies the axiom $\Abfs{d}$. By \cref{prop-sing-fp}, we see that $\D^\sing(\Acal)$ is equal to the quotient of $\D^b(\Acal)$ by the subcategory consisting of all objects of finite projective dimension. 

In another words, we obtain $\D^\sing(\Acal)$ as the quotient of $\D^b(\Acal)$ over the thick subcategory:
$$\{X \in \D^b(\Acal) \mid \exists n \in \Zbb ~\forall Y \in \Acal ~\forall i>n: \Hom_{\D^b(\Acal)}(X,Y[i]) = 0\}.$$ 
In this viewpoint, we can make a direct comparison to Orlov's definition of a singularity category of a general triangulated category, as defined in \cite[Definition 1.7]{orlo-06}. Applied to $\D^b(\Acal)$, Orlov's definition yields a quotient over the following thick subcategory, defined by a formula in which the quantifiers appear in a slightly different order:
$$\{X \in \D^b(\Acal) \mid \forall Y \in \Acal ~\exists n \in \Zbb ~\forall i>n: \Hom_{\D^b(\Acal)}(X,Y[i]) = 0\}.$$ 
Finally, we recall that Orlov proved in \cite[Proposition 1.11]{orlo-06} that the two subcategories on display coincide in the case $\Acal = \coh X$ where $X$ is a separated noetherian scheme of finite Krull dimension with enough locally free sheaves. On the other hand, there exists an essentially small abelian category $\Acal$ with a set of projective generators for which these two subcategories do not coincide, see \cite[Example 3.3]{Zh18}.
\section{Locally coherent Grothendieck categories from restrictable $t$-structures}\label{sec:restrictable}

In this section we focus on a particular source of Grothendieck categories which satisfy
the hypothesis of \cref{theorem:recollement}, that is, they are locally coherent
and their derived category is compactly generated. We are going to consider the hearts of intermediate t-structures in the derived category of commutative noetherian ring which are induced by a cotilting complex. In the following two subsections we set the scene by gathering the necessary concepts and results. Then we proceed to characterize the cotilting t-structures whose heart is locally coherent as those satisfying a restrictability condition. The notion of objects of type $\FPinf$ of the previous section will again play an important role here.

\subsection{Compactly generated and restrictable $t$-structures in $\D(\Mod{R})$}

A $t$-structure $\Tbb = (\Ucal,\Vcal)$ is \newterm{compactly generated} if there
is a set $\Scal \subseteq \Tcal^c$ such that $\Vcal = \Scal\Perp{0}$, or
equivalently, if $\Vcal = (\Ucal \cap \Tcal^c)\Perp{0}$.

Alonso Tarrío, Jeremías López and Saorín \cite{AJS10} showed that compactly
generated $t$-structures admit a full classification in geometric terms in the
case $\Tcal = \D(\Mod{R})$, the unbounded derived category of a commutative
noetherian ring $R$. Let $\Spec{R}$ denote the Zariski spectrum of $R$. A subset
$V$ of $\Spec{R}$ is called \newterm{specialization closed} if $V$ is a union of
Zariski-closed sets (equivalently, $V$ is an upper subset of the poset
$(\Spec{R},\subseteq)$). An \newterm{sp-filtration} of $\Spec{R}$ is an
order-preserving function $\Phi: \mathbb{Z} \rightarrow 2^{\Spec{R}}$ such that
$\Phi(n)$ is a specialization closed subset for each $n \in \mathbb{Z}$.

\begin{thm}\emph{(\cite[Theorem 3.10]{AJS10})}\label{thm-AJS}
	Let $R$ be a commutative noetherian ring. There is a bijective
	correspondence between sp-filtrations $\Phi$ of $\Spec{R}$ and the set of
	compactly generated $t$-structures in $\D(\Mod{R})$. The bijection assigns
	to $\Phi$ a $t$-structure with the aisle $\Ucal_\Phi$ defined as follows:
	\[\Ucal_\Phi = \{X \in \D(\Mod{R}) \mid \Supp H^n(X) \subseteq \Phi(n)
		~\forall n \in \mathbb{Z}\}.\]
\end{thm}

\begin{dfn}
	Let $\Gcal$ be a locally coherent Grothendieck category, that is, a locally
	finitely presented Grothendieck category such that the full subcategory $\fp
	\Gcal$ of finitely presented objects forms an abelian subcategory of
	$\Gcal$. A $t$-structure $\Tbb = (\Ucal,\Vcal)$ in $\D(\Gcal)$ is
	\newterm{restrictable} if the
	pair $(\Ucal \cap \D^b(\fp \Gcal),\Vcal \cap \D^b(\fp \Gcal))$ is a
	$t$-structure in $\D^b(\fp \Gcal)$.
\end{dfn}

Under mild assumption on a commutative noetherian ring $R$, the restrictability
of a compactly generated $t$-structure in $\D(\Mod{R})$ can be read rather
directly from the associated sp-filtration. For the definition of a pointwise
dualizing complex we refer the reader e.g. to \cite[\S~6]{AJS10}; in particular,
any (classical) dualizing complex is a pointwise dualizing complex.

\begin{thm}\emph{(\cite[Corollary 4.5, Theorem 6.9]{AJS10})}\label{weak-Cousin}
	Let $R$ be a commutative noetherian ring and let $\Tbb$ be the compactly
	generated $t$-structure corresponding to an sp-filtration $\Phi$. Consider
	the following two conditions:
	\begin{itemize}
		\item[(i)] $\Tbb$ is restrictable (to $\D^b(\mod R)$),
		\item[(ii)] $\Phi$ satisfies the \newterm{weak Cousin condition}, that
			is, whenever $\pp \subsetneq \qq$ are prime ideals such that $\qq$
			is minimal over $\pp$, then for any $n \in \mathbb{Z}$ the
			implication $\qq \in \Phi(n) \Rightarrow \pp \in \Phi(n-1)$ holds.
	\end{itemize}
	Then $(i) \Rightarrow (ii)$ holds. Furthermore, if $R$ admits a pointwise
	dualizing complex then also $(ii) \Rightarrow (i)$ holds.
\end{thm}

\begin{rmk}\label{ubiqutous}
	Let $R$ be a commutative noetherian ring. Restrictable $t$-structures in
	$\D(\Mod R)$ are ubiquitous:
	\begin{itemize}[wide=0pt]
		\item \cite[Theorem 2.16, Remark 2.7]{PV2} A Happel-Reiten-Smal\o~ (HRS)
			$t$-structure obtained from a hereditary torsion pair in $\Mod R$ is
			compactly generated and restrictable. In view of \cref{thm-AJS},
			these $t$-structures correspond to sp-filtrations $\Phi$ such that
			$\Phi(n) = \Spec R$ for all $n<0$ and $\Phi(n) = \emptyset$ for all
			$n>0$.
		\item Assume that $\mathsf{d}$ is a codimension function on $\Spec R$,
			that is, a function $\mathsf{d}: \Spec R \rightarrow \mathbb{Z}$
			such that $\mathsf{d}(\qq) = \mathsf{d}(\pp) + 1$ whenever $\pp
			\subsetneq \qq$ are primes with $\qq$ minimal over $\pp$. Then the
			assignment $\Phi_\mathsf{d}(n) = \{\pp \in \Spec R \mid
			\mathsf{d}(\pp) > n\}$ defines an sp-filtration which satisfies the
			weak Cousin condition.
		
			Furthermore, any pointwise dualizing complex $D$ induces a
			codimension function $\mathsf{d}_D$ \cite[p. 287]{Ha}, and therefore
			a restrictable $t$-structure, see \cite[\S6]{AJS10}.
		\item  Following \cite[\S6.4]{AJS10}, if $R$ admits a dualizing complex
			$D$, the restrictable $t$-structure induced by the codimension
			function $\mathsf{d}_D$ has a particularly nice description. The
			functor $\RHom_R(-,D)$ induces a duality functor on the category
			$\D^b(\mod R)$, and therefore it sends the standard $t$-structure to
			another $t$-structure on $\D^b(\mod R)$, called the
			\newterm{Cohen-Macaulay $t$-structure}. This $t$-structure then
			naturally lifts to a restrictable $t$-structure in $\D(\Mod R)$, see
			\cite[\S3]{MZ}, and coincides with the compactly generated
			$t$-structure corresponding to the sp-filtration
			$\Phi_{\mathsf{d}_D}$.
	\end{itemize}
\end{rmk}

\subsection{Silting and cosilting $t$-structures and their realization
	functors}\label{sec:silting}

Let $\Tcal$ be a triangulated category and $M \in \Tcal$. We define the full
subcategories $M\Perp{>0}=\{X \in \Tcal \mid \Hom_\Tcal(M,X[i]) ~\forall i >
0\}$ and $\Perp{>0}M=\{X \in \Tcal \mid \Hom_\Tcal(X,M[i]) ~\forall i > 0\}$,
the subcategories $M\Perp{\leq 0}, M\Perp{< 0}$ and $\Perp{\leq 0}M, \Perp{<
0}M$ are defined analogously.

Following Psaroudakis-Vitória \cite{PV} and Nicolás-Saorín-Zvonareva \cite{NSZ},
an object $T$ in $\Tcal$ is \newterm{silting} if the pair
$(T\Perp{>0},T\Perp{\leq 0})$ is a $t$-structure in $\Tcal$, which we call a
\newterm{silting $t$-structure}. A silting object $T$ (as well as the induced
$t$-structure) is called \newterm{tilting} if $\Add(T) \subseteq T\Perp{<0}$,
where $\Add(T)$ is the smallest full subcategory of $\Tcal$ containing $T$ and closed
under all coproducts and retracts. Dually, an object $C \in \Tcal$ is
\newterm{cosilting} if the pair $(\Perp{\leq 0}C,\Perp{>0}C)$ is a $t$-structure
in $\Tcal$, which we call a \newterm{cosilting $t$-structure}. A cosilting
object $C$ (as well as the induced $t$-structure) is called \newterm{cotilting}
if $\Prod(C) \subseteq \Perp{<0}C$, where $\Prod(C)$ is the smallest subcategory
of $\Tcal$ containing $C$ and closed under all products and retracts.

(Co)silting and (co)tilting objects serve to study triangle equivalences, often
induced by the realization functors associated to the induced (co)silting
$t$-structures. Let us specialize now to the case $\Tcal = \D(\Gcal)$, where
$\Gcal$ is a Grothendieck category. Given a (co)silting object $M \in \D(\Gcal)$
denote the heart of the silting $t$-structure $\Tbb_M$ by $\Hcal_M$ and the
induced realization functor as $\real_M: \D^b(\Hcal_M) \rightarrow \D(\Gcal)$.
We call a (co)silting object in $\D(\Gcal)$ \newterm{bounded} if the induced
(co)silting $t$-structure is intermediate. Recall that the intermediacy
implies that the realization functor factors through $\D^b(\Gcal)$. Specializing
the result of Psaroudakis and Vitória to Grothendieck categories, we have the
following tilting theorem.

\begin{thm}\emph{\cite[Corollary 5.2]{PV}}\label{T:PV}
	Let $\Gcal$ be a Grothendieck category and $M \in \D(\Gcal)$ a bounded
	(co)silting object. Then $\real_M: \D^b(\Hcal_M) \rightarrow \D^b(\Gcal)$ is
	a triangle equivalence if and only if the object $M$ is (co)tilting.
\end{thm}

We remark that if $T$ is a silting object then $\Hcal_T$ is an abelian category
with a projective generator \cite{PV}. If $T$ is (additively equivalent to) a
compact object of $\D(\Gcal)$ then it follows that $\Hcal_T$ is equivalent to a
module category, and if in addition $T$ is tilting then we have $\Hcal_T \cong
\Mod{\End_{\D(\Gcal)}(T)}$ \cite[Corollary 4.7]{PV}. On the other hand, consider
a module category $\Mod R$ and a bounded cosilting object $C \in \D(\Mod{R})$.
Then it is known that the heart $\Hcal_C$ is a Grothendieck category
\cite[Proposition~3.10]{MV}. 

In \cite{Vi}, Virili extended the (co)tilting realization functors to the
unbounded level by constructing realization equivalences of standard derivators.
See also the formulation \cite[Theorem E]{Vi} characterizing restrictable
derived equivalences.

\begin{thm}\emph{\cite[Theorem C, D]{Vi}}\label{T:Vi}
	Let $\Gcal$ be a Grothendieck category and $M \in \D(\Gcal)$ a bounded
	tilting (resp. cotilting) object. Then there is an
	equivalence $\mathfrak{real}_M: \mathfrak{D}_{\Hcal_M} \rightarrow
	\mathfrak{D}_\Gcal$ of derivators which is bounded.
\end{thm}

In the situation of Theorem~\ref{T:Vi}, we denote the triangle equivalence on
the base as $\real_M := \mathfrak{real}_M^\star$. Then the triangle equivalence
$\real_M: \D(\Hcal_M) \rightarrow \D(\Gcal)$ is an unbounded realization functor
\cite[Theorem 7.7, Theorem 7.9]{Vi} which restricts to a bounded realization
functor $\D^b(\Hcal_M) \rightarrow \D^b(\Gcal)$ which is an equivalence.

% An important fact for us is that the equivalence of derivators from item
% $(iv)$ is bounded (cf. \cite[Theorem 7.9]{Vi}).
%\begin{lemma}\label{lem:bounded-der} In the situation of \cref{T:PVVi}, the
%equivalence $\mathfrak{real}_C:\D_{\Hcal_C} \rightarrow \D_{\Mod{R}}$ of
%derivators is bounded.  \end{lemma} \begin{proof} To check that for any small
%category $I$, the triangle equivalence $\mathfrak{real}_C^I:
%\D(\Hcal_C^I)\rightarrow \D(\Mod{R}^I)$ restricts to an equivalence
%$\D^b(\Hcal_C^I)\rightarrow \D^b(\Mod{R}^I)$, it is enough to see that the
%standard cohomology of all objects of the heart $\Hcal_C$ inside $\D(\Mod{R})$
%is uniformly bounded. Indeed, let $\Xcal \in \D^b(\Hcal_C^I)$. Then the
%cohomology of all the coordinates $(\Xcal)_i$ for $i \in I$ are uniformly
%bounded. Since $\mathfrak{real}_C$ is an equivalence of derivators, we obtain
%from \cref{square-derivator} that $(\mathfrak{real}_C^I(\Xcal))_i \cong
%\mathfrak{real}_C^\star(\Xcal_i)$ in $\mathfrak{D}_{\Mod R}(\star) = \D(\Mod
%R)$ for all $i \in I$, and thus we see that $\mathfrak{real}_C^I(\Xcal) \in
%\D^b(\Mod{R}^I)$.

%Finally, the fact that $\Hcal_C$ inside $\D(\Mod{R})$ has uniformly bounded
%cohomology follows directly from the fact that the $t$-structure $\Tbb_C$ is
%intermediate.  \end{proof}

A compilation of known results gives a nice characterization of $t$-structures
in $\D(\Mod{R})$ induced by bounded cotilting objects amongst all intermediate
$t$-structures when $R$ is commutative noetherian.

\begin{thm}\emph{(\cite{PV,HN})}\label{theorem:cotilting-real}
	Let $R$ be a commutative noetherian ring and $\Tbb$ an intermediate
	$t$-structure in $\D(\Mod{R})$. The following are equivalent:
	\begin{enumerate}[label=(\roman*)]
		\item there is a triangle equivalence $\D(\Hcal_\Tbb) \rightarrow
			\D(\Mod{R})$ which restricts to the bounded level and $\Hcal_\Tbb$
			is a locally finitely presented Grothendieck category,			
		\item the realization functor $\real^b_{\Tbb}: \D^b(\Hcal_\Tbb)
			\rightarrow \D^b(\Mod{R})$ is an equivalence and $\Hcal_\Tbb$ is a
			Grothendieck category,
		\item $\Tbb$ is a cotilting $t$-structure. 
	\end{enumerate}
\end{thm}

\begin{proof}
	$(i) \Rightarrow (ii):$ Clear.

	$(ii) \Rightarrow (iii):$ This is \cref{T:PV}.

	$(iii) \Rightarrow (i):$ The first part follows by \cref{T:Vi}. Since $R$ is
	commutative noetherian, it is is known that $\Tbb_\Tbb$ is a compactly
	generated $t$-structure \cite{Hcom,HN} and the heart $\Hcal_\Tbb$ is
	a locally finitely presentable Grothendieck category \cite{SS20}. 
\end{proof}

Finally, we record a recently established strong connection between the
cotilting property and restrictability of the associated $t$-structure.

\begin{thm}[{\cite[Corollary~6.18]{PV2}}]
	Let $R$ be a commutative noetherian ring and $\Tbb$ be an intermediate,
	compactly generated, and restrictable $t$-structure in $\D(\Mod{R})$. Then
	$\Tbb$ is a cotilting $t$-structure.
\end{thm}
\subsection{Cotilting t-structures with a locally coherent heart}
Recall from \cref{theorem:cotilting-real} that if $R$ is a commutative
noetherian ring and $\Tbb$ is an intermediate cotilting $t$-structure, then
$\Tbb$ is compactly generated and $\Hcal$ is a locally finitely presentable
Grothendieck category by \cite[Theorem 1.6]{SS20}. As mentioned,
we are mostly interested in the case when $\Hcal$ is in addition
locally coherent: therefore, in this section we consider the following setting.

\begin{setting}\label{setting:cotilting}
	Let $R$ be a commutative noetherian ring. Let $\Tbb_C$ be a
	$t$-structure, with heart $\Hcal_C$, such that:
	\begin{enumerate}[label=(C\arabic*)]
		\item\label{item:1-cotilt} $\Tbb_C$ is the cotilting
			$t$-structure associated to a cotilting object $C$.
		\item\label{item:1-bounded} $\Tbb_C$ is intermediate.
		\item\label{item:1-restrict} $\Hcal_C$ is a locally coherent
			Grothendieck category.
	\end{enumerate}
\end{setting}

Condition \ref{item:1-bounded} is equivalent to the requirement that
$C\in\K^b(\inj R)$, which is sometimes included in the definition of a cotilting
object. The fact that $C$ is cotilting provides us with a triangle equivalence
\[\real_C\colon \D(\Hcal_C)\to \D(\Mod{R})\]
which restricts to the level of bounded derived categories and which lifts to an
equivalence between the standard derivators, see \cref{sec:silting}. In
particular, it will ensure that $\D(\Hcal_C)\simeq\D(\Mod R)$ is compactly
generated.

The main goal of this subsection is to characterize \cref{setting:cotilting} using
the restrictability of the $t$-structure $\Tbb_C$. To do that, we first need to
better understand the compact objects in the bounded derived category of
$\Hcal_C$. Recall from \cref{rem:inclusion-bounded-compact} that we have an
inclusion $\D^b(\fp{\Hcal_C}) \subseteq \D^b(\Hcal_C)^c$. We will use the
derived equivalence to $\Mod R$ to show that this inclusion is an equality.

\begin{lemma}\label{lemma:fpinf-der-eq}
	Let $\Gcal$ and $\Ecal$ be Grothendieck categories and $\eta:
	\mathfrak{D}_\Gcal \rightarrow \mathfrak{D}_{\Ecal}$ a bounded equivalence
	of derivators. Then an object $X \in \D^b(\Gcal)$ is of type $\FPinf$ if
	and only if $\eta^\star(X)$ is of type $\FPinf$ in $\D^b(\Ecal)$.
\end{lemma}

\begin{proof}
	Let $I$ be a directed small category and $\Ycal \in \D^b(\Gcal^I)$. Then
	there is the following commutative square induced by application of the
	equivalence $\eta$ between derivators, where all of the maps are the
	naturally induced ones:
	\[\begin{tikzcd}[column sep=1pc]
		\varinjlim_{i \in I}\Hom_{\D(\Gcal)}(X,\Ycal_i)
			\arrow{d} \arrow{r}{\cong} &
		\varinjlim_{i \in I}\Hom_{\D(\Ecal)}(\eta^\star X,(\eta^I \Ycal)_i)
			\arrow{d}\\
		\Hom_{\D(\Gcal)}(X,\hocolim_I \Ycal) \arrow{r}{\cong} &
		\Hom_{\D(\Ecal)}(\eta^\star X,\hocolim_I (\eta^I \Ycal))
	\end{tikzcd}\]
	Note that both the horizontal isomorphisms are induced by the triangle
	equivalence $\eta^\star$. Indeed, this follows from the two canonical
	isomorphisms induced by the derivator equivalence $\eta$:
	\[\hocolim_I (\eta^I \Ycal) \cong \eta^\star(\hocolim_I \Ycal) \text{  and
		} (\eta^I \Ycal)_i \cong \eta^\star(\Ycal_i),\]
	see \cref{square-derivator-hocolim} and \cref{square-derivator}. Since the
	equivalence $\eta$ is bounded, $\eta^I \Ycal \in \D^b(\Ecal^I)$. Therefore,
	if $\eta^\star$ is of type $\FPinf$ then the right vertical map is an
	isomorphism by \cref{lemma:hcompact-1}. Then the square implies that the
	left vertical map is an isomorphism for any choice of $\Ycal \in
	\D^b(\Gcal^I)$, and so $X$ is of type $\FPinf$. The converse implication
	follows similarly using the fact that $\eta^\star$ and $\eta^I$ are
	equivalences between the bounded derived categories.
\end{proof}

\begin{lemma}\label{lemma:H2}
	In \cref{setting:cotilting}, we have $\D^b(\Hcal_C)^c=\D^b(\fp{\Hcal_C})$.
	In particular, the derived equivalence $\real_C$ restricts to a triangle equivalence
	$\D^b(\fp{\Hcal_C}) \rightarrow \D^b(\mod R)$.
\end{lemma}

\begin{proof}
	Recall from \cref{T:Vi} that the cotilting $t$-structure $\Tbb$ induces a
	bounded equivalence $\mathfrak{real}_C: \D_{\Hcal_C} \rightarrow
	\D_{\Mod{R}}$ of derivators. In particular, we have a triangle equivalence
	$\D^b(\Hcal_C) \toeq \D^b(\Mod R)$ obtained by restriction of
	$\mathfrak{real}_C^\star: \D(\Hcal_C) \toeq \D(\Mod R)$. Then
	$\mathfrak{real}_C^\star$ further restricts to an equivalence
	$\D^b(\Hcal_C)^c \toeq \D^b(\Mod R)^c$ between the categories of compact
	objects. Since $R$ is noetherian, $\D^b(\Mod R)^c = \D^b(\mod R)$ by
	\cite[Corollary~6.17]{Rou}, and $\D^b(\mod R)$ is also precisely the
	subcategory of $\D^b(\Mod{R})$ consisting of objects of type $\FPinf$, see
	\cref{lemma:fp-equal-hcompact}. Then \cref{lemma:fpinf-der-eq} applies and
	shows that $\D^b(\Hcal_C)^c$ coincides with the subcategory of all objects
	of type $\FPinf$ of $\D^b(\Hcal_C)$. But by \cref{lemma:fp-equal-hcompact}
	this is precisely the subcategory $\D^b(\fp{\Hcal_C})$. 

	Finally, note that we proved the second statement along the way, since
	$\real_C = \mathfrak{real}_C^\star$.
\end{proof}

\begin{cor}\label{cor:sing-eq}
	In \cref{setting:cotilting}, the functor $\real_C$ induces a triangle
	equivalence $\D^\sing(\Hcal_C) \to \D^\sing(\Mod R)$ between singularity
	categories.
\end{cor}
\begin{proof}
	By \cref{lemma:H2}, the derived equivalence $\real_C: \D(\Hcal_C) \to \D(\Mod
	R)$ restricts to an equivalence $\D^b(\fp{\Hcal_C}) \rightarrow \D^b(\mod
	R)$. Since $\real_C$ also restricts to an equivalence $\D(\Hcal_C)^c \to
	\D(\Mod R)^c$ between the subcategories of compact objects, the result
	follows formally by passing to Verdier quotients.
\end{proof}

Now we are ready to formulate the main result of this section, that is, to
characterize the case in which the heart $\Hcal_C$ is a locally coherent
category. Our result can be seen as a refinement of the restrictability
characterization due to Marks and Zvonareva \cite[Corollary 4.2]{MZ} in the
special case of intermediate compactly generated $t$-structures in
$\D(\Mod{R})$.

\begin{thm}\label{thm:cotilting-side}
	Let $R$ be a commutative noetherian ring and $\Tbb$ be an intermediate
	compactly generated $t$-structure in $\D(\Mod R)$, with heart $\Hcal$. Then
	the following are equivalent:
	\begin{enumerate}[label=(\roman*)]
		\item[(i)] we are in \cref{setting:cotilting}, that is, $\real^b_\Tbb$
			is an equivalence and $\Hcal$ is locally coherent;
		%\item $C$ is an elementary cogenerator;\todo{define}
		\item[(ii)] the $t$-structure $\Tbb$ restricts to $\D^b(\mod R)$.
		%\item the $t$-structure $\Tbb$ is compactly generated and restrictable to $\D^b(\mod R)$. 
	\end{enumerate}
\end{thm}
\begin{proof}
	%The equivalence $(i) \iff (ii)$ is \cite[Theorem 5.12]{Laking}. 
	%The relation between $(iii)$ and $(iv)$ is established in \cite[Corollary
	%4.5, Corollary 6.10]{AJS10}.

	Recall that $\real^b_\Tbb$ being an equivalence amounts to $\Tbb$ being
	induced by a cotilting object $C$ by \cref{theorem:cotilting-real}, and
	therefore the description in $(i)$ indeed corresponds to
	\cref{setting:cotilting}.

	The two claims of the implication $(ii) \Rightarrow (i)$ are proven in \cite[Corollary~6.17]{PV2} and \cite[Theorem~6.3]{Saorin}, respectively.

	It remains to show $(i) \Rightarrow (ii)$. Assume now that $\Hcal$ is
	locally coherent. To establish that $\Tbb$ is restrictable, we just need to
	recall from \cref{lemma:H2} that the derived equivalence $\real_C: \D(\Hcal)
	\toeq \D(\Mod{R})$ restricts to an equivalence $\D^b(\fp \Hcal) \toeq
	\D^b(\mod R)$. The $t$-structure $\Tbb$ corresponds under $\real$ to the
	standard $t$-structure on $\D(\Hcal)$, which clearly restricts to a
	$t$-structure in $\D^b(\fp \Hcal)$.
\end{proof}
\cref{thm:cotilting-side} admits the following reformulation in terms of cosilting objects.
\begin{cor}\label{cor:cotilt-coherent}
	Let $R$ be a commutative noetherian ring and $C \in \D(\Mod R)$ a bounded cosilting object with the induced cosilting t-structure $\Tbb_C = (\Perp{\leq 0}C,\Perp{>0}C)$ and heart $\Hcal_C$. The following are equivalent:
	\begin{enumerate}[label=(\roman*)]
		\item[(i)] $C$ is cotilting and $\Hcal_C$ is locally coherent,
		%\item $C$ is an elementary cogenerator;\todo{define}
		\item[(ii)] $\Tbb_C$ is restrictable.
		%\item the $t$-structure $\Tbb$ is compactly generated and restrictable to $\D^b(\mod R)$. 
	\end{enumerate}
\end{cor}

\begin{ex}[Cotilting heart that is not locally coherent]
	Let $(R,\mm)$ be a local Cohen-Macaulay ring of Krull dimension at least 2. Consider an sp-filtration $\Phi$ given as follows:
	$$\Phi(n) = \begin{cases} \Spec R & \text{for } n<0; \\ \{\mm\} & \text{for } n=0,1; \\ \emptyset & \text{for } n>1. \end{cases}$$
	Since the grade of the ideal $\mm$ is at least 2 by the assumption, this filtration corresponds to a bounded cosilting complex $C$ whose cohomology is concentrated in degree zero (so that $C$ is isomorphic in $\D(\Mod R)$ to a \newterm{cotilting module}), see \cite[Theorem 4.2]{APST14} and the discussion in \cite[Remark 5.10]{HNS}. In particular, $C$ is a cotilting object. On the other hand, the sp-filtration clearly does not satisfy the weak Cousin condition. By \cref{weak-Cousin}, the induced cotilting t-structure is not restrictable and so the cotilting heart is not locally coherent by \cref{cor:cotilt-coherent}.
\end{ex}

\begin{ex}[Locally coherent cosilting heart that is not cotilting]
	As in \cite[Example~6.19]{PV2}, assume $R$ has Krull dimension $1$, and
	consider the sp-filtration $\Phi$ given by:
	\[ \Phi(n) =
	\begin{cases}
		\Spec R&\text{for }n<0;\\
		\mathsf{Max} (R)&\text{for }n=0,1;\\
		\emptyset &\text{for }n>1.
	\end{cases}
	\]
	Then the associated heart is $\Hcal_\Phi=\Ccal[1]\oplus \Tcal[-1]$, where
	$\Tcal\subseteq\Mod R$ is the hereditary torsion class associated to
	$\mathsf{Max}(R)$ and $\Ccal\subseteq\Mod R$ is the corresponding Giraud
	subcategory, $\Ccal=\Tcal^{\bot_{0,1}}$. $\Hcal_\phi$ is the product of two locally noetherian
	Grothendieck category, so it is locally noetherian (in particular, locally
	coherent). On the other hand, the $t$-structure does not restrict (because
	$\Phi$ does not satisfy the weak Cousin condition), and therefore it is not
	cotilting.
\end{ex}

\begin{cor}\label{cor:restr-recollement}
	Let $R$ be a commutative noetherian ring, and $\Tbb$ an intermediate,
	compactly generated, restrictable $t$-structure, with heart $\Hcal$. Then
	Krause's recollement exists for $\Hcal$.
\end{cor}

\begin{proof}
	By \cref{thm:cotilting-side}, $\Tbb$ fits in \cref{setting:cotilting}, and
	so $\Hcal$ satisfies the hypothesis of \cref{theorem:recollement}.
\end{proof}

As another application of \cref{lemma:H2}, we can show that the two versions of
coderived categories of $\Hcal_C$ due to Becker and Positselski coincide. Recall
that an object $M \in \Hcal_C$ is \newterm{fp-injective} if
$\Ext_{\Hcal_C}^1(F,M) = 0$ for all $F \in \fp {\Hcal_C}$. Furthermore, $M \in
\Hcal_C$ is of \newterm{finite fp-injective dimension} if $M$ is isomorphic in
$\D(\Hcal_C)$ to a bounded complex of fp-injective objects concentrated in
non-negative degrees.

\begin{lemma}\label{lemma:fpinjfiniteinj}
	In \cref{setting:cotilting}, any object in $\Hcal_C$ of finite fp-injective
	dimension is of finite injective dimension.
\end{lemma}

\begin{proof}
	Let $M \in \Hcal_C$, put $X = \real_C(M) \in \D^b(\Mod{R})$, and let us
	denote the converse equivalence to $\real_C$ as $\real_C^{-1}: \D(\Mod{R})
	\toeq \D(\Hcal_C)$. Since the $t$-structure $\Tbb$ is intermediate, and
	using \cref{lemma:H2}, there is an integer $n \in \Zbb$ such that
	$\real_C^{-1}(\mod{R}) \subseteq \D(\fp{\Hcal_C}) \cap \D^{\geq n}$.  If $M$
	is of finite fp-injective dimension then
	$\Hom_{\D(\Hcal_C)}(\D(\fp{\Hcal_C})^{\geq n},M[i]) = 0$ for all $i \gg 0$.
	Applying $\real_C$ we therefore obtain $\Hom_{\D(\Mod{R})}(\mod{R},X[i]) =
	0$ for all $i \gg 0$, which amounts to $X \in \D^b(\Mod{R})$ being of finite
	injective dimension in $\D(\Mod{R})$, since $R$ is noetherian. Equivalently,
	we have $\Hom_{\D(\Mod{R})}(\D(\Mod{R})^{\geq 0},X[i]) = 0$ for $i \gg 0$.
	But using the intermediacy of $\Tbb$ again, we know that $\real_C \Hcal_C[j]
	\subseteq \D(\Mod R)^{\geq 0}$ for $j \ll 0$, and so it follows by applying
	$\real^{-1}_C$ that $\Hom_{\D(\Hcal_C)}(\Hcal_C,M[i+j]) = 0$ for $i+j \gg
	0$, which in turn implies that $M$ is of finite injective dimension in
	$\Hcal_C$.
\end{proof}

\begin{cor}\label{cor:poscoderived}
	In \cref{setting:cotilting},  the coderived category $\K(\inj{\Hcal_C})$ (in
	Becker's sense) is equivalent to the coderived category in Positselski's
	sense.
\end{cor}

\begin{proof}
	This follows directly from \cite[\S3.7, Theorem]{Pos11} in view of
	\cref{lemma:fpinjfiniteinj} .
\end{proof}

We finish this section with an example of a locally coherent Grothendieck category
which does not satisfy \cite[Hypothesis 7.1]{St} even though its derived
category is compactly generated. In fact, we obtain it as a heart $\Hcal_{\Tbb}$
in $\D(\Mod{R})$ induced by a compactly generated, intermediate and restrictable
$t$-structure.

\begin{ex}\label{example:hyp7.1}
	Let $(R,\mathfrak m)$ be a commutative and noetherian local ring, of
	dimension 1, which is not Cohen-Macaulay; for example, take $R$ to be the
	localisation of $k[x,y]/(x^2,xy)$, for an algebraically closed field $k$, at
	the maximal ideal $\mathfrak m=(x,y)$.  In particular, we have
	$1=\mathsf{dim}(R)>\mathsf{depth}(R)=0$; and then, by the
	Auslander-Buchsbaum formula, every non-zero finitely generated module is
	projective or has infinite projective dimension (in other words, the small
	finitistic global dimension of $R$ is $0$). Moreover, since $R$ is not
	Cohen-Macaulay, $\mathfrak m$ is an associated prime of $R$; and the other
	primes are minimal, so they are associated as well, i.e. $\Ass R=\Spec R$.
	Therefore, every cyclic module $R/\mathfrak p R$ for a prime $\mathfrak p$
	is a subobject of a projective module ($R$ itself). It follows from Matlis'
	Theorem and \cite[Theorem 7.1]{Bass} that the finitistic injective global
	dimension of $R$ and, by duality, also the finitistic weak global dimension
	of $R$ are $0$. We recall that this means that any $R$-module of finite flat
	dimension is automatically flat.

	Let $V=\{\mathfrak m\}$, consider the associated hereditary torsion pair
	$\mathbf{t}=(\Tcal,\Fcal)$ in $\Mod R$, and let $\Hcal$ be the HRS-tilt of
	$\Mod R$ with respect to $\mathbf{t}$; namely, $\Hcal=\Fcal[1]\ast\Tcal$, we
	refer to \cite{PV2} for terminology and details. Notice that since
	$\D(\Hcal)\cong \D(\Mod{R})$ (by \cite[Corollary 5.11]{PV2}) the former is
	compactly generated. Also, $\Hcal$ is the heart of the Happel-Reiten-Smal\o~
	$t$-structure corresponding to the torsion pair $(\Tcal,\Fcal)$, and this is
	an intermediate $t$-structure which is compactly generated and restrictable
	(\cite[Remark 4.8 and Theorem 2.16(3)]{PV2}).

	Nonetheless, we shall show that there are no non-zero finitely presented
	objects of finite projective dimension in $\Hcal$, and therefore
	\cite[Hypothesis 7.1]{St} is not satisfied.
	
	Since $R$ has dimension 1, every subset of $\Spec R$ is coherent, and
	therefore $V$ corresponds to a flat ring epimorphism $R\to S$, see \cite[\S 5.2]{AMSVT}; given our
	choice of $V$, $S$ will be a regular ring of dimension $0$.  In $\Hcal$,
	there is a hereditary torsion pair $\mathbf{s}=(\Tcal,\Mod S[1])$ (see
	\cite[\S4.2]{PV2}).

	Let $X$ be a finitely presented object of $\Hcal$, i.e.
	$X\in\Hcal\cap\D^b(\mod R)$, and assume it has finite projective dimension. Note that this implies that $X$ is of finite projective dimension also as an object of $\D(\Mod R)$.
	Consider its approximation sequence with respect to $\mathbf{s}$ in $\Hcal$,
	i.e. the triangle \[T\rightarrow X \rightarrow L[1] \rightarrow T[1]\] with
	$T\in\Tcal$ and $L$ an $S$-module. In particular, since
	$\text{gl.dim}(S)=0$, $L$ is a projective $S$-module; since $S$ is a flat
	$R$-module, it has finite projective dimension over $R$ \cite[Seconde
	partie, Corollaire 3.2.7]{RG}, and then so does $L$. From the triangle
	above, we deduce that $T$ has finite projective dimension as well.  Then,
	its flat dimension in $\Mod R$ is also finite, and since the finitistic weak
	global dimension of $R$ is $0$, $T$ is a flat $R$-module. Now, we claim that
	this implies $T=0$. Indeed, consider a presentation
	\[0\rightarrow K\rightarrow F\rightarrow T\rightarrow 0\]
	with $F=R^{(\alpha)}$ a free $R$-module. Since $T$ is flat, this
	sequence is pure exact, and therefore the torsion radical $t$ of
	$\mathbf{t}$ gives a short exact sequence
	\[0\rightarrow tK\rightarrow tF\rightarrow T\rightarrow 0.\]
	By construction, $tR$ is supported on $V=\{\mathfrak m\}$, and since
	it is finitely generated, this means that $V(\ann{tR})=\{\mathfrak
	m\}$. Hence $\mathfrak m=\sqrt{\ann{tR}}$, and since $R$ is
	noetherian it follows that there exists $n$ such that $\mathfrak m^n
	tR=0$. Therefore, $tR, tF=(tR)^{(\alpha)}$ and also $T$ are $R/\mathfrak
	m^n$-modules. $T$ is also flat over $R/\mathfrak m^n$, and since this is
	an artinian local ring, $T$ is free, i.e. $T\cong (R/\mathfrak
	m^n)^{(\beta)}$.  But then, if $T\neq 0$, its direct summand
	$R/\mathfrak m^n$ should be a finitely presented
	flat $R$-module, and therefore projective, which is a contradiction
	because it would force $R$ to be artinian (and therefore
	$0$-dimensional).
	
	It follows that our finitely presented object $X$ of $\Hcal$ is
	isomorphic to $L[1]$. But then, $L$ is a finitely presented $R$-module of
	finite projective dimension, hence it is projective, hence free. Now,
	since $L$ is also an $S$-module, if $L\neq0$ this would imply that
	$R\in\Mod{S}$. In particular, $R$ would be torsion-free with respect to
	$\mathbf{t}$, which is not the case since $\mathfrak m\in\Ass{R}$.
	We conclude that $X\cong L[1]=0$.
\end{ex}

\section{The equivalence of recollements}\label{subsection:situation}

Let again $R$ be a commutative noetherian ring. In this last section we compare
the recollements arising from \cref{cor:restr-recollement} with Krause's
recollement for $\Mod R$.

By \cref{thm:cotilting-side},
\cref{setting:cotilting} characterizes the case in which we have an intermediate
compactly generated restrictable $t$-structure $\Tbb$.
Consider now the following seemingly new situation.

\begin{setting}\label{setting:tilting}
	Let $\Hcal$ be a locally coherent Grothendieck category, and assume that
	there exists an object $T$ in $\D(\Hcal)$ such that:
	\begin{enumerate}[label=(T\arabic*)]
		\item\label{item:2-tilt} $T$ is compact tilting.
		\item\label{item:2-pdim} $T$ has finite projective dimension, i.e.
			$\Hom_{\D(\Hcal)}(T,\Hcal[i])=0$ for $i\gg 0$, % (cf. \cref{rem:fin-proj}).
		\item\label{item:2-end} $\End_{\D(\Hcal)}(T)$ is isomorphic to the commutative
			noetherian ring $R$.
	\end{enumerate}
\end{setting}

Condition \ref{item:2-tilt} ensures that $\D(\Hcal)$ is compactly generated.
Since $T$ is compact, it belongs to $\D^b(\Gcal)$, see
\cref{prop:compacts-are-bounded}. Under this assumption, similarly to before,
condition \ref{item:2-pdim} is equivalent to requiring the
tilting $t$-structure $\Tbb_T$ of $\D(\Hcal)$ associated to $T$ to be
intermediate. Conditions \ref{item:2-tilt} and \ref{item:2-end} imply that its
heart $\Hcal_T$ is isomorphic to $\Mod R$, and we have a triangle equivalence
\[\real_T\colon \D(\Mod{R})=\D(\Hcal_T)\to \D(\Hcal).\]

Using the equivalences $\real_C$ and $\real_T$, we see that these two settings
are the two sides of the same picture: starting from Setting
\ref{setting:cotilting}, the choices $\Hcal:=\Hcal_C$ and $T:=\real_C^{-1}(R)$
fit Setting \ref{setting:tilting}; conversely, taking $C:=\real_T^{-1}(W)$ for
an injective cogenerator $W$ of $\Hcal$, one obtains the $t$-structure $\Tbb_C$
as the pullback along $\real_T$ of the standard $t$-structure of $\D(\Hcal)$.
In the following we will work with \cref{setting:tilting}, while
\cref{setting:cotilting} will serve as motivation.

Our goal is to construct an equivalence between the two recollements
\[\begin{tikzcd}[column sep=1pc]
	\Kac(\inj \Hcal) \arrow{r} &
		\K(\inj \Hcal) \arrow[shift left]{l} \arrow[shift right]{l} \arrow{r} &
		\D(\Hcal) \arrow[shift left]{l} \arrow[shift right]{l} &
	\Kac(\inj R) \arrow{r} &
		\K(\inj R) \arrow[shift left]{l} \arrow[shift right]{l} \arrow{r} &
		\D(\Mod{R}) \arrow[shift left]{l} \arrow[shift right]{l}
\end{tikzcd}\]
In order
to do that, we replace the derived equivalence $\real_C$ by another one which we
are able to lift to the coderived level. We start by fixing a convenient
resolution of $T$.

\begin{lemma}
	Up to shift, $T$ admits a resolution $T:=(F_{-n}\to F_{-n+1}\to
	\cdots\to F_0)$ with finitely presented objects $F_i\in\fp\Hcal$.
\end{lemma}

\begin{proof}
	Since $T$ is compact, by \cref{prop:compacts-are-bounded} it belongs to
	$\D^b(\fp\Hcal)$, so it is quasi-isomorphic to a complex over the
	abelian category $\fp\Hcal$. By taking soft truncations this complex can
	be made strictly bounded.
\end{proof}

Now we consider the functor
$\HOM(T,-)\colon \C(\Hcal)\to \C(\Zbb)$,
defined as the totalization of the bicomplex $\HOM^{\bullet,\bullet}(T,-)$.
Notice that this bicomplex is always bounded along the direction of $T$ (because
we chose a strictly bounded resolution of $T$).

Since $R$ is commutative, $\D(\Hcal)\cong\D(\Mod{R})$ is an $R$-linear category, and
then so is $\Hcal$. 
The bicomplex $\HOM^{\bullet,\bullet}(T,-)$ and its totalization $\HOM(T,-)$
have therefore terms in $\Mod R$ and $R$-linear differentials; this gives us a functor
\begin{equation}
	\Psi:= \HOM(T,-)\colon \C(\Hcal)\to \C(R).
\end{equation}

Moreover, if $X\in\C(\Hcal)$ is contractible, then the rows of
$\HOM^{\bullet,\bullet}(T,X)$ are also contractible, since $\Hom_\Hcal(F_i,-)$
is an additive functor for all $-n\leq i\leq 0$. It follows that
$\HOM(T,X)\in\C(R)$ is also contractible, which gives us a functor
\begin{equation}
	\Psi:=\HOM(T,-)\colon \K(\Hcal)\to \K(R).
\end{equation}
In particular, by restriction of the domain, $\Psi$ induces functors on the
subcategories $\K(\inj\Hcal)\subseteq\K(\fpinj\Hcal)\subseteq\K(\Hcal)$, which
we will continue to denote by $\Psi$.

We record immediately that $\Psi$ induces a derived equivalence
$\D(\Hcal)\cong\D(\Mod{R})$.

\begin{lemma}\label{lemma:rhom-equivalence}
	The functor $\RHom_\Hcal(T,-):= Q\Psi Q_r\colon \D(\Hcal)\to \D(\Mod{R})$ is
	an equivalence. Moreover, it restricts to an equivalence $\D^b(\Hcal)\to
	\D^b(\Mod{R})$, and also to an equivalence $\D^b(\fp \Hcal)\to
	\D^b(\mod{R})$.
\end{lemma}

\begin{proof}
	By \ref{item:2-tilt} and \ref{item:2-end} we have $\RHom_\Hcal(T,T)\cong
	\End_{\D(\Hcal)}(T)\cong R$, so the functor $\RHom_\Hcal(T,-)$ sends a
	compact generator of $\D(\Hcal)$ to a compact generator of $\D(\Mod{R})$.
	Moreover, since $\RHom_\Hcal(T,-)$ is $R$-linear on $\Hom$-sets, it must
	induce the isomorphism $\End_{\D(\Hcal)}(T)\cong \End_R(R)=R$ of
	endomorphism rings. Since $T$ is a compact generator of $\D(\Hcal)$ and $R$
	is a compact generator of $\D(\Mod R)$, a standard arguments shows that
	$\RHom_\Hcal(T,-)$ induces an equivalence $\D(\Hcal)^c \toeq \D(\Mod R)^c$
	between the categories of compact objects (see e.g.
	\cite[Proposition~6]{Miy}).  Lastly, $\RHom_\Hcal(T,-)$ preserves
	coproducts, since $T$ is compact.  Then, the derived equivalence is
	established by double \textit{dévissage} (\cref{lemma:double-devissage}).

	For the claim about the bounded equivalence, let $X\in\D(\Hcal)$. Then its
	image $\RHom_\Hcal(T,X)$ belongs to $\D^b(\Mod{R})$ if and only if
	$\Hom_{\D(\Hcal)}(T,X[i])=0$ for all but finitely many $i\in\Zbb$; and this
	means that $X$ has finitely many cohomologies with respect to $\Tbb_T$.
	Since $\Tbb_T$ is intermediate, this is equivalent to $X$ belonging to
	$\D^b(\Hcal)$. Therefore, $\RHom_\Hcal(T,-)$ restricts to an equivalence
	$\D^b(\Hcal) \toeq \D^b(\Mod R)$, and therefore also to an equivalence
	$\D^b(\Hcal)^c \toeq \D^b(\Mod R)^c$ between compact objects of the bounded
	derived categories. By \cref{lemma:H2} and \cite[Corollary~6.17]{Rou}, this last equivalence identifies with
	the desired equivalence $\D^b(\fp \Hcal)\toeq \D^b(\mod{R})$.
\end{proof}

\begin{lemma}\label{lemma:psi-coproducts}
	The functor $\Psi\colon \C(\Hcal)\to \C(R)$ preserves direct limits (and in
	particular coproducts). Therefore, also the induced functor $\Psi\colon
	\K(\Hcal)\to \K(R)$ and its restriction $\Psi\colon \K(\fpinj\Hcal)\to
	\K(R)$ preserve coproducts.
\end{lemma}

\begin{proof}
	Coproducts in $\K(\Hcal)$ are computed termwise, as in $\C(\Hcal)$.
	Moreover, since $\fpinj\Hcal$ is closed under coproducts in $\Hcal$,
	coproducts in $\K(\fpinj\Hcal)$ are computed as in $\K(\Hcal)$. It is
	then enough to prove the claim for $\Psi\colon \C(\Hcal)\to \C(R)$.

	Now, let $X_\alpha:=(\cdots\to X_\alpha^i\to X_\alpha^{i+1}\to
	\cdots)\in\C(\Hcal)$ be a direct system of objects, and consider their
	direct limit $\varinjlim X_\alpha=(\cdots\to \varinjlim X_\alpha^i\to
	\varinjlim X_\alpha^{i+1}\to \cdots)$. $\Psi$ sends it to the
	totalization of the bicomplex
	\[\begin{tikzcd}[sep=1pc]
		\cdots\arrow{r} &\Hom_\Hcal(F_{0},\varinjlim X_\alpha^i) \arrow{r}\arrow{d} &
			\Hom_\Hcal(F_{0},\varinjlim X_\alpha^{i+1}) \arrow{r}\arrow{d} &
			\cdots\\
		\cdots\arrow{r} &\Hom_\Hcal(F_{-1},\varinjlim X_\alpha^i) \arrow{r}\arrow{d} &
			\Hom_\Hcal(F_{-1},\varinjlim X_\alpha^{i+1}) \arrow{r}\arrow{d} &
			\cdots\\
		&\vdots \arrow{d} & \vdots \arrow{d}  \\
		\cdots\arrow{r} &\Hom_\Hcal(F_{-n},\varinjlim X_\alpha^i) \arrow{r} &
			\Hom_\Hcal(F_{-n},\varinjlim X_\alpha^{i+1}) \arrow{r}&
			\cdots
	\end{tikzcd}\]
	Since the $F_i$'s are finitely presented in $\Hcal$, the functors
	$\Hom_\Hcal(F_i,-)$ commute naturally with the direct limits, so
	$\HOM^{\bullet,\bullet}(T,\varinjlim X_\alpha)$ is
	isomorphic in $\C(\C(\Hcal))$ to the direct limit of the bicomplexes
	$\HOM^{\bullet,\bullet}(T,X_\alpha)$.
	Totalization also commutes with direct limits, and so $\Psi$ preserves
	them.
\end{proof}

In order to obtain a functor between the coderived categories, we want $\Psi$ to
preserve coacyclicity. Recall that a locally finitely presentable Grothendieck
category $\Gcal$ admits a natural notion of a pure exact sequence, and that a
complex in $\C(\Gcal)$ is \newterm{pure-acyclic} if it is acyclic and in
addition, each exact sequence $0 \rightarrow Z^i(X) \rightarrow X^i \rightarrow
Z^{i+1}(X) \rightarrow 0$ induced by the cocycles is pure exact in $\Gcal$.

We start by recalling the following fact.

\begin{prop}[\cite{St}]\label{lemma:pureacyclics<coacyclics}
	Over a locally coherent Grothendieck category, pure-acyclic complexes are
	coacyclic.
\end{prop}

\begin{proof}
	This follows mainly from \cite[\S6.2]{St}; we recollect the argument
	for the convenience of the reader. Let $\Gcal$ be a
	locally coherent Grothendieck category, and $X$ a complex in
	$\C(\Gcal)$. Then $X$ corresponds to a coacyclic object of $\K(\Gcal)$
	if and only if it is $\Ext^1_\C$-orthogonal to $\C(\inj\Gcal)$, i.e. if
	it belongs to the left class of the functorially complete cotorsion pair
	generated by disks of $\fp\Gcal$. Now, this left class is closed under
	retracts and transfinite extensions, and pure-acyclic
	complexes are (retracts of) transfinite extensions of disks of
	$\fp\Gcal$ in $\C(\Gcal)$ by \cite[Lemma 5.6]{St}.
\end{proof}

\begin{lemma}\label{lemma:psi-coacyclics}
	The restriction $\Psi\colon \K(\fpinj\Hcal)\to \K(R)$ preserves
	coacyclic complexes.
\end{lemma}
\begin{proof}
	As a partial converse of \cref{lemma:pureacyclics<coacyclics}, a complex
	$X\in\C(\Hcal)$ of fp-injectives is coacyclic in $\K(\Hcal)$ if and only
	if it is pure-acyclic \cite[Proposition~6.11]{St}.
	By \cite[Lemma 4.14]{St}, a complex $X$ in $\C(\Hcal)$ is pure-acyclic
	if and only if it is a direct limit of bounded contractible complexes.
	Since $\Psi\colon\C(\Hcal)\to \C(R)$ preserves
	both direct limits (\cref{lemma:psi-coproducts}) and contractibility,
	$\Psi(X)$ will also be pure-acyclic by the same characterisation.
	Then we conclude by \cref{lemma:pureacyclics<coacyclics} that $\Psi(X)$
	is also coacyclic.
\end{proof}

In view of the equivalences
\[ \begin{tikzcd}
	\K(\inj\Hcal) \arrow{r}{\subseteq}[swap]{\cong} &
	\K(\fpinj\Hcal)/\{\text{pure acyclics}\} \arrow{r}[swap]{\cong} & \D^\co(\Hcal)
\end{tikzcd}\]
by \cref{lemma:psi-coacyclics} we deduce that $\Psi$ induces a functor
\begin{equation}
	\Rbb^\co\Psi\colon \D^\co(\Hcal)\to \D^\co(R).
\end{equation}
On an object $X\in\D^\co(\Hcal)$, $\Rbb^\co\Psi$ is computed by first resolving
$X$ by a complex of fp-injectives (or even injectives), then applying $\Psi$ and
considering the resulting complex as an object of $\D^\co(R)$. When identifying
$\D^\co(\Hcal)\cong \K(\inj \Hcal)$ and $\D^\co(R)\cong \K(\inj R)$,
$\Rbb^\co\Psi$ is then the composition
\begin{equation}\label{eq:rcopsi-injectives}
	\begin{tikzcd}
		\Rbb^\co\Psi\colon\quad \K(\inj\Hcal) \arrow{r}{\subseteq} & \K(\Hcal) \arrow{r}{\Psi} &
			\K(R) \arrow{r}{I_\lambda} & \K(\inj R).
	\end{tikzcd}
\end{equation}

\begin{prop}\label{prop:rcopsi-equivalence}
	$\Rbb^\co\Psi\colon \D^\co(\Hcal)\to \D^\co(R)$ is an equivalence.
\end{prop}

\begin{proof}
	We want to argue by double \textit{dévissage}.

	First, $\Rbb^\co\Psi\colon \D^\co(\Hcal)\to \D^\co(R)$ preserves
	coproducts, since $\Psi$ does (\cref{lemma:psi-coproducts}).

	Now we show that it induces an equivalence between the subcategories of
	compact objects.  In view of the identification $\D^\co(\Hcal)\cong
	\K(\inj\Hcal)$, a compact object of $\D^\co(\Hcal)$ is identified with
	the dg-injective resolution $X$ of an object in $\D^b(\fp\Hcal)$; in
	particular, this is a bounded below complex. When we apply $\Psi$ and
	then $I_\lambda$, as in \eqref{eq:rcopsi-injectives}, we obtain again a
	bounded below complex, first in $\K(R)$ and then in $\K(\inj R)$. This
	last object $Y:=I_\lambda\Psi(X)$, in particular, is a dg-injective
	complex.  Since we have $X\cong Q_rQX$ and $Y\cong Q_rQY$ in
	$\K(\inj\Hcal)$ and $\K(\inj R)$, respectively, we can write
	\[ \Rbb^\co\Psi(X)=Y\cong Q_rQY=Q_rQI_\lambda\Psi X=Q_rQ\Psi X\cong
		Q_rQ\Psi Q_rQX=:(\ast) \]
	Now, by definition, $\RHom_\Hcal(T,-):=Q\Psi Q_r$, so we can continue
	\[ (\ast)=Q_r\RHom_\Hcal(T,QX)\]
	It is therefore sufficient to show that $Q_r\RHom_\Hcal(T,Q-)$ is an
	equivalence between $\K(\inj\Hcal)^c$ and $\K(\inj R)^c$. Now, $Q\colon
	\K(\inj\Hcal)^c\to \D^b(\fp\Hcal)$ and $Q_r\colon \D^b(\mod R)\to K(\inj
	R)^c$ are equivalences, and $\RHom_\Hcal(T,-)\colon \D^b(\fp\Hcal)\to
	\D^b(\mod R)$ is an equivalence by \cref{lemma:rhom-equivalence}.
\end{proof}

Now that we have the equivalence between the coderived categories, we show that it
preserves the recollements. First we need a technical lemma.

\begin{lemma}\label{lemma:T-is-in-Ql}
	$I_\lambda T\cong Q_lQI_\lambda T$ in
	$\D^\co(\Hcal)\cong\K(\inj\Hcal)$.
\end{lemma}

\begin{proof}
	Let $E$ be the dg-injective resolution of $T$; we have a triangle in
	$\K(\Hcal)$
	\[A\to T\to E\to A[1]\]
	with $A$ acyclic. Since $T$ is bounded below, $E$ and then $A$ are also
	bounded below. $A$ is therefore also coacyclic. This means that
	$E\cong I_\lambda T$ in $\K(\inj\Hcal)$.
	Now, since $E$ is dg-injective we have $E\cong Q_rQE\cong Q_rQT$; but
	$QT$ is compact in $\D(\Hcal)$, and therefore $Q_rQT\cong Q_lQT$ by \cref{lem-Ql-Qr}. We
	conclude as wanted that $I_\lambda T\cong Q_lQ T$ in $\K(\inj\Hcal)$.
\end{proof}

\begin{lemma}\label{lemma:rcopsi-acyclics}
	$\Rbb^\co\Psi\colon \D^\co(\Hcal)\to \D^\co(R)$ preserves acyclics.
\end{lemma}

\begin{proof}
	Identifying $\D^\co(\Hcal)\cong\K(\inj\Hcal)$ and in view of
	\eqref{eq:rcopsi-injectives}, let $X\in\K(\inj\Hcal)$
	be acyclic. For every $n\in\Zbb$ we have
	\begin{multline*}
		H^nI_\lambda\Psi X\cong H^n\Psi X=H^n\HOM(T,X)\cong\\
			\cong\Hom_{\K(\Hcal)}(T,X[n])\cong
				\Hom_{\K(\inj\Hcal)}(I_\lambda
				T,X[n])\overset{(1)}{\cong}\\
			\cong \Hom_{\K(\inj\Hcal)}(Q_lQT, X[n])\cong
			\Hom_{\D(\Hcal)}(QT,QX[n])\overset{(2)}{=}0
	\end{multline*}	
	where $(1)$ is by \cref{lemma:T-is-in-Ql} and $(2)$ because $QX=0$.
\end{proof}

%\begin{lemma}\label{lemma:rcopsi-Ql}
%	Identifying $\D^\co(\Hcal)\cong\K(\inj\Hcal)$, $\Rbb^\co\Psi$ restricts
%	to a functor between the images of $Q_l\colon \D(\Hcal)\to \K(\inj\Hcal)$
%	and $Q_l\colon \D(\Mod{R})\to \K(\inj R)$.
%\end{lemma}
%\begin{proof}
%	Let us first focus on the image under $Q_l$ of the subcategory of
%	compact objects of $\D(\Hcal)$. Since
%	$Q_l(\D(\Hcal)^c)=Q_r(\D(\Hcal)^c)\subseteq Q_r(\D^b(\fp\Hcal))$, the
%	proof of \cref{prop:rcopsi-equivalence} shows that the restriction of
%	$\Rbb^\co\Psi$ to this subcategory is the fully faithful
%	functor
%	\[ \Rbb^\co\Psi= Q_r\RHom_\Hcal(T,Q-) \colon Q_l(\D(\Hcal)^c)\to
%	Q_r(\D^b(\mod R))\]
%	Moreover, since $\RHom_\Hcal(T,-)\colon \D(\Hcal)\to \D(\Mod{R})$ is an
%	equivalence, it restricts to an equivalence on compacts of $\D(\Hcal)$;
%	this is enough for $\Rbb^\co\Psi$ to induce an equivalence
%	\[Q_l(\D(\Hcal)^c)\to Q_r(\D(\Mod{R})^c)=Q_l(\D(\Mod{R})^c).\]
%	Now, $\Rbb^\co\Psi$ preserves coproducts, so it sends the localising
%	subcategory generated by $Q_l(\D(\Hcal)^c)$ to the localising
%	subcategory generated by $Q_l(\D(\Mod{R})^c)$. These are $Q_l(\D(\Hcal))$
%	and $Q_l(\D(\Mod{R}))$ respectively.
%\end{proof}

\begin{thm}\label{theorem:eor}
	The functor $\Rbb^\co\Psi\colon \D^\co(\Hcal)\to \D^\co(R)$ induces an equivalence
	of recollements, that is, there is a diagram
	\[\begin{tikzcd}[column sep=1pc, row sep=1pc]
		\DS(\Hcal) \arrow{d}[swap]{\mathbb{S}\Psi}{\cong} \arrow{r} &
			\D^{\co}(\Hcal) \arrow{d}[swap]{\Rbb^\co\Psi}{\cong} \arrow[shift left]{l} \arrow[shift right]{l} \arrow{r} &
			\D(\Hcal) \arrow{d}[swap]{\Rbb\Psi}{\cong} \arrow[shift left]{l} \arrow[shift right]{l} \\
		\DS(\Mod{R}) \arrow{r} &
			\D^\co(\Mod{R}) \arrow[shift left]{l} \arrow[shift right]{l} \arrow{r} &
			\D(\Mod{R}) \arrow[shift left]{l} \arrow[shift right]{l}
	\end{tikzcd}\]
	in which the rows are the recollements from \cref{rem:stablederived} of $\Hcal$ and $\Mod{R}$ and such that all the six obvious squares commute.
\end{thm}

\begin{proof}
	Identify $\D^\co(\Hcal)\cong\K(\inj\Hcal)$ and $\D^\co(R)\cong\K(\inj
	R)$. By \cref{prop:rcopsi-equivalence}, $\Rbb^\co\Psi$ is an equivalence. By
	\cref{lemma:rcopsi-acyclics}, it preserves acyclicity.
	%By \cref{lemma:rcopsi-Ql} it sends objects in $Q_l(\D(\Hcal))$
	%to objects in $Q_l(\D(\Mod{R}))$.
	%Now, the pairs of subcategories $(\{\text{image of }Q_l\},\{\text{acyclics}\})$ are
	%$t$-structures in $\K(\inj\Hcal)$ and $\K(\inj R)$, and this allows to
	%conclude that the restrictions
	%\[\Rbb^\co\Psi\colon \Kac(\inj\Hcal)\to \Kac(\inj R)\qquad
	 % \Rbb^\co\Psi\colon Q_l(\D(\Hcal))\to Q_l(\D(\Mod{R}))\]
	%are equivalences. In turn, this means also that $\Rbb^\co\Psi$
	%restricts to an equivalence between the subcategories of dg-injectives,
	%as these are characterized by being right orthogonal to acyclics. It follows that we have a commutative square
	In view of basic results on recollement equivalences (see
	\cref{sec:recollements}), it is enough to show that the following square is
	commutative up to equivalence
	\[ \begin{tikzcd}
		\K(\inj \Hcal) \arrow{d}{\cong}[swap]{\Rbb^\co\Psi}\arrow{r}{Q} & \D(\Hcal) \arrow{d}{\cong}[swap]{\Rbb\Psi}\\
		\K(\inj R) \arrow{r}{Q} & \D(\Mod{R})
	\end{tikzcd}\]
	where $\Rbb\Psi = \RHom_\Hcal(T,-)$. Since $\Rbb^\co\Psi$ preserves
	acyclics, the composition $Q\Rbb^\co\Psi$ kills objects from
	$\Kac(\inj{\Hcal})$, and thus the approximation triangle with respect to the
	stable $t$-structure $(\Kac(\inj \Gcal),Q_r(\D(\Gcal))$ in $\K(\inj \Gcal)$
	yields a natural equivalence $Q\Rbb^\co\Psi \cong Q\Rbb^\co\Psi Q_rQ$. Then
	we can compute similarly as in \cref{prop:rcopsi-equivalence}:
	\[Q\Rbb^\co\Psi Q_rQ = Q I_\lambda \Psi Q_rQ \cong Q\Psi Q_rQ = \Rbb\Psi Q.\]
	The rest follows by denoting the induced triangle equivalence $\DS(\Hcal)
	\rightarrow \DS(\Mod{R})$ by $\mathbb{S}\Psi$.
\end{proof}

\bibliographystyle{plain}
\bibliography{bibitems}
\end{document}